%% file: main_revised_with_changes_clean.tex
\documentclass[11pt]{article}

\usepackage[utf8]{inputenc}

\newcommand{\Tau}{\mathcal{T}}

\usepackage{tikz}
\usepackage{amsthm}
\usepackage{amsmath} 
\usepackage{amssymb} 
\usepackage{bm}
\usepackage{multirow}
\usepackage{booktabs}
\usepackage{graphicx}
\usepackage{hyperref}
\usepackage{placeins}

\graphicspath{ {./figs/}} 

\newtheorem{thm}{Theorem}[section]
\newtheorem{prp}[thm]{Proposition}
\newtheorem{lem}[thm]{Lemma}

\newtheorem{asp}[thm]{Assumption}

\theoremstyle{definition}

\theoremstyle{remark}

\newcommand{\grad}{\mathop{\rm grad}\nolimits}
\renewcommand{\div}{\mathop{\rm div}\nolimits}

\makeatletter
\renewcommand{\@biblabel}[1]{#1.\hfill}
\makeatother

\newcommand{\corr}[1]{{\color{black}#1}}

\begin{document}
\title{Meshfree Generalized Multiscale Exponential Integration Method for Parabolic Problems}
\author{Djulustan Nikiforov\thanks{North-Eastern Federal University, Yakutsk, Russia; \texttt{dju92@mail.ru}}  \and Leonardo A. Poveda\thanks{Department of Mathematics, The Chinese University of Hong Kong, Hong Kong Special Administrative Region; \texttt{lpoveda@math.cuhk.edu.hk}} \and  Dmitry Ammosov\thanks{North-Eastern Federal University, Yakutsk, Russia; \texttt{dmitryammosov@gmail.com}}  \and Yesy Sarmiento\thanks{Escuela de Matemáticas y Ciencias de la Computación, Universidad Nacional Autónoma de Honduras, Boulevard Suyapa, Tegucigalpa, M.D.C., Honduras; \texttt{yesy.sarmiento@unah.hn}}\and Juan Galvis\thanks{Departamento de Matemáticas, Universidad Nacional de Colombia, Carrera 45 No. 26-85, Edificio Uriel Gutiérrez, Bogotá D.C., Colombia; \texttt{jcgalvisa@unal.edu.co}}
}

\maketitle

\begin{abstract}
This paper considers flow problems in multiscale heterogeneous porous media. The multiscale nature of the modeled process significantly complicates numerical simulations due to the need to compute huge and ill-conditioned sparse matrices, which negatively affect both the computational cost and the stability of the numerical solution. We propose a novel combined approach of the meshfree Generalized Multiscale Finite Element Method (MFGMsFEM) and exponential time integration for solving such problems. MFGMsFEM provides a robust and efficient spatial approximation, allowing us to consider complex heterogeneities without constructing a coarse computational grid. At the same time, using the cost-effective MFGMsFEM matrix, exponential integration provides a robust temporal approximation for stiff multiscale problems, allowing larger time steps. For the proposed multiscale approach, we provide a rigorous convergence analysis, including the new analysis of the MFGMsFEM spatial approximation. We conduct numerical experiments to computationally verify the proposed approach by solving linear and semi-linear flow problems in multiscale media. 
Numerical results demonstrate that the proposed multiscale method achieves significant reductions in computational cost and improved stability, even with larger time steps, confirming the theoretical analysis.

\end{abstract}


\section{Introduction}

Flow simulation in porous media involves modeling the movement of fluids through natural or artificial materials such as soil, rocks, and filters 
\cite{narayanan2018flow,battiato2019theory, bastian2013simulation}. This field provides valuable insights into various phenomena, including groundwater flow, which is crucial for managing water resources, predicting contamination spread, and assessing human impact on aquifers; oil reservoir behavior, helping optimize production strategies, estimate recovery rates, and predict reservoir performance; subsurface carbon storage, essential for addressing climate change by evaluating the feasibility and safety of carbon capture and storage projects; radioactive waste management, critical for designing safe repositories and assessing long-term containment; and the design of porous materials, which allows for the computation of optimal application designs by modeling their elastic, diffusion related, and transport properties.

Modeling diffusion in heterogeneous porous materials presents significant challenges, primarily due to the variability and complexity of the permeability coefficient \(\kappa(x)\). This coefficient represents how easily fluid moves through the porous media. The way \( \kappa \) is described and how it variates in space and time will determine the real computational complexity and numerical analysis challenges of the whole simulation. Other parameters still need to be considered, e.g., boundary conditions and forcing terms, but we focus on the permeability tensor. In applications, the description of \( \kappa \) can be achieved in several ways depending on modeling goals and restrictions. The alternatives are:

\begin{itemize}
    \item Oracle form: A code or black box can return the permeability value at selected locations. This black-box computation can be complemented with information related to the computational complexity associated with the computation of the permeability.
    \item Parameterized formula: In some applications, a parameterized formula is available to compute the permeability at selected locations. This parameterized formula may have been obtained previously by adjusting some data or models. 
    \item Discrete description in a mesh: In many practical applications, such as subsurface flow problems, the permeability coefficient is described using a mesh of the domain, and we have a constant tensor value for each element in the mesh. We mention the following cases:
    \begin{itemize}
        \item Classical meshes: In this case, the mesh corresponds to a conforming triangulation made of squares or triangles in two dimensions or tetrahedral or parallelepipeds in three dimensions. In this case, we can use classical finite element methods to compute solutions of flow problems.
        \item General polygonal meshes: In this case, the mesh is not standard, and its elements are general polygonal subdomains, e.g., Voronoi meshes. In this case, nonstandard techniques, such as spatial discontinuous Galerkin formulation of the virtual element method, have to be employed to compute solutions to flow problems.
    \end{itemize}
\end{itemize}
The reader must also consider the fact that the values of the permeability may not be available in the whole physical domain of the problem but only in special subsets of representative volumes. Also, the permeability value may be uncertain due to the lack of information in some applications. 

Various approximations of solutions to the flow problem in porous media and related questions have been explored. We mention the references such as \cite{abreu_recursiveparabolic_2021, Huang_exponetial_parabolic_2019, con2020, SUN201876}. Our focus is on the numerical computation of solutions to this problem. Classical methods face challenges with robustness and efficiency when dealing with high-contrast multiscale coefficients. We address three key challenges in designing efficient numerical methods for these problems; the first two are:

\begin{enumerate}
    \item The spatial resolution needed for accurate approximation of the solution is related to the smallest scale at which coefficient $\kappa$ variations occur. Additionally, discontinuities and high jumps of the coefficient pose additional difficulties in numerical approximation, affecting accuracy and efficiency. See \cite{egh12, efendiev2011multiscale, abreu2019convergence, calo2016randomized} for further details.

    \item The presence of high-contrast coefficients, even without complex multiscale variations, reduces the stability region of time discretization methods such as Crank–Nicolson and similar integrators. Extensive research has been devoted to addressing these challenges, but we do not intend to provide a comprehensive review here.
\end{enumerate}

Challenge 1 affects both time-dependent and time-independent problems. Classical multiscale methods offer good approximations only for moderate-to-low-contrast coefficients. However, Generalized Multiscale Finite Element Methods (GMsFEM) were developed to handle problems with high-contrast coefficients, primarily utilizing local eigenvalue problems to construct appropriate coarse-mesh approximation spaces. 
GMsFEM has proven to be effective in numerically upscaling partial differential equations with multiscale high-contrast parameters, finding applications in linear heat transfer, fluid flow, elasticity, and parabolic equations, among other practical problems \cite{chung2014generalized, gao2015generalized, fu2019generalized, chung2014adaptive}. 
For further details, see \cite{egh12, efendiev2011multiscale, abreu2019convergence, calo2016randomized, zambrano2021fast}.

Now, we turn to the third key challenge.  When the permeability description is available only on a complicated fine grid, it is very costly to construct appropriate coarse meshes where GMsFEM methodology can be applied. This difficulty has led researchers to propose several alternatives:

\begin{itemize}
    \item Extension of GMsFEM methodologies to complicated coarse meshes (such as the ones obtained by software like ParMETIS). See \cite{calvo2024robust}. Here, the main ingredient is the construction of a multiscale partition of unity functions in the resulting complicated mesh. 
    \item Extension of GMsFEM methodologies to meshfree GMsFEM method in 
    \cite{nikiforov2023meshfree} that have the advantage of not requiring a coarse mesh construction but only a point cloud of coarse nodes. 
\end{itemize}

We focus on the coarse meshfree GMsFEM developed in \cite{nikiforov2023meshfree}.
In \cite{nikiforov2023meshfree}, one of the authors designed a coarse meshfree Generalized Multiscale Finite Element Method (MFGMsFEM). The coarse space is constructed based on an existing fine-scale computational grid that resolves the heterogeneous parameters of the problem. The method builds upon the GMsFEM, incorporating multiscale basis functions to account for the heterogeneous parameters on the coarse scale. These basis functions are generated offline using local spectral problems. Numerical solutions for both two-dimensional and three-dimensional problems are presented in this work. The MFGMsFEM has been extended to several problems, including nonlinear and time-dependent problems \cite{djulustan2024meshfree, nikiforov2023meshfree, nikiforov2024meshfree,nikiforov2024meshfree1}.

In the present paper, we consider the case of approximation of solutions to time-dependent problems posed in high-contrast multiscale media. We also need to address stability issues \cite{ contreras2023exponential, Expint2007,poveda2024second}, and several alternatives have been proposed in the literature depending on the application and computational resources available. Here, we mention: 
a) partially explicit time integration, as proposed in 
\cite{li2022partially} and extended to MFGMsFEM in 
\cite{djulustan2024meshfree}, and 
b) the exponential time integration, a method known for its effectiveness in handling stiff problems,  
that can take full advantage of what GMsFEM has to offer in terms of dimension reduction by being able to use larger time steps in the simulations (see 
\cite{contreras2023exponential, poveda2024second}).

\corr{We focus on the exponential integration strategies. In the referred works, the exponential integration was used for classical GMsFEM constructions.  
This paper proposes employing exponential integration in combination with MFGMsFEM methods to show that stability and accuracy can be maintained even with larger time steps for the (coarse) meshfree GMsFEM method. Combining GMsFEM for spatial approximation and exponential integration for temporal discretization offers a promising solution to the challenges posed by high-contrast multiscale coefficients. This approach enhances computational efficiency while ensuring stability and accuracy in simulations. 
Unlike traditional methods that rely on coarse computational grids or standard temporal discretizations, our approach leverages the strengths of MFGMsFEM and exponential integration. This combination handles complex heterogeneities without needing coarse grids and provides a robust temporal approximation for stiff multiscale problems.
We show that this combination can also be implemented where no coarse mesh is available in the computation of the coarse space. In this case, a stable method is important for applications where constructing a coarse mesh is impractical or impossible.}

The rest of the paper is organized as follows. Section \ref{sec:preliminaries} provides a short review of the related work and existing methods for solving parabolic problems in multiscale media, highlighting their limitations and the need for the proposed approach. Section \ref{sec:meshfree_gmsfem} details the meshfree generalized multiscale finite element method (MFGMsFEM). Section \ref{sec:gmsfem_time_discretization} describes the exponential time integration method used for temporal discretization, explaining its advantages and implementation details. Section \ref{sec:convergence} presents the convergence analysis for the proposed method. Section \ref{sec:numerical_results} discusses the setup of the numerical experiments conducted to validate the proposed method, comparing the proposed method with traditional approaches in terms of accuracy, computational efficiency, and stability. Finally, Section 7 concludes the paper by summarizing the key findings, discussing the implications of the results, and suggesting potential directions for future research.

\section{Preliminaries}\label{sec:preliminaries}
This section briefly introduces the mathematical model considered in this paper. Then, we outline the spatial approximation using the finite element method and the temporal discretization using the finite difference method. We will consider this solution method as a reference in numerical experiments.

\subsection{Mathematical model}\label{subsec:mathematical_model}
We consider the fluid filtration process in a multiscale porous medium with high contrast. The mathematical model is described by a semilinear parabolic equation for the pressure field $p$, and it is written as
\begin{equation}\label{eq:fine_sys}
\begin{aligned}
\frac{\partial p}{\partial t} - \div \left( \kappa(x) \grad p \right) &= f(p), &&x \in \Omega \times I,\\
p &= p_D, && x \in \partial \Omega \times I,\\
p(0, x) & = p_0(x), &&x \in \Omega,
\end{aligned}\end{equation} 
where $\Omega \subset \mathbb{R}^2$ is \corr{a domain} with boundary 
$\partial \Omega$, $I=[0,T]$ denotes the time domain, $\kappa (x)$ is a multiscale high-contrast heterogeneous permeability coefficient, such that $\kappa_{\min}\leq \kappa(x)\leq \kappa_{\max}$, where $0<\kappa_{\min}<\kappa_{\max}$. $f(p)$ is a nonlinear source term, $p_D$ is a boundary value, and $p_0$ is an initial value. The methods implemented here can be easily extended to the three-dimensional case. 

\subsection{Fine-grid finite element approximation}\label{subsec:fine_grid_fem}
Let us assume that $p(t) \in V$ for all $t \in I$, where $V = \{ v \in H^1(\Omega): v|_{\partial \Omega} = p_D\}$. Moreover, we define the function space $\widehat{V} = \{ v \in H^1(\Omega): v|_{\partial \Omega} = 0\}$. Then, we obtain the following variational formulation of the problem \eqref{eq:fine_sys}: Find $p(t) \in V$ such that
\begin{equation}\label{eq:var_form}
m\left(\frac{\partial p}{\partial t}, v\right) + a(p, v) = F(p; v), \quad \text{for all } v \in \widehat{V},
\end{equation}
where the forms $m$, $a$, and $F$ are defined as follows
\begin{equation}\label{eq:fine_forms}
m(p, v) = \int_\Omega p v dx, \quad
a(p, v) = \int_\Omega \kappa(x) \grad p \cdot \grad v dx \quad \mbox{ and }
F(p; v) = \int_{\Omega}f(p) v dx.
\end{equation}

Next, we partition the computational domain $\Omega$ into fine cells $K^h \in \Tau^h$, where $\Tau^h$ is a fine grid. We assume that the fine-grid size $h > 0$ is small enough to resolve all the fine-scale heterogeneities of the permeability coefficient $\kappa(x)$.

We define the finite-dimensional function spaces $V^h \subset V$ and $\widehat{V}^h \subset \widehat{V}$ with respect to $\Tau^h$. Therefore, we obtain the following semi-discrete variational problem: Find $p(t) \in V_h$ such that
\begin{equation}\label{eq:semidiscrete_var_form}
\begin{split}
m\left(\frac{\partial p}{\partial t}, v\right) + a(p, v) = F(p; v), \quad &\text{for all } v \in \widehat{V}_h, \quad t \in I,\\
(p(0), v) = (p_0, v), \quad &\text{for all } v \in \widehat{V}_h,
\end{split}
\end{equation}
where $p = \sum_{i=1}^{N_v^h} p_i \phi_i$, and $\phi_i$ are two-dimensional fine-scale piecewise linear basis functions.


Note that we can represent this problem in the following continuous time matrix form: Find the vector $p(t) = [p_1(t), p_2(t), ..., p_{N_v^h}(t)]^T$ such that
\begin{equation}\label{eq:cont_time_matrix_form}
\begin{split}
M \frac{d p}{d t} + A p &= b(p),\\
M p(0) &= \widehat{p},
\end{split}
\end{equation}
where
\begin{equation}\label{eq:fine_matrices_and_vectors}
M = [m_{ij}], \quad m_{ij} = m(\phi_j, \phi_i), \quad
A = [a_{ij}], \quad a_{ij} = a(\phi_j, \phi_i), \quad
b = [b_{i}], \quad b_{i} = F(p; \phi_i).
\end{equation}

\subsection{Temporal discretization using finite difference method}\label{subsec:fine_grid_fd}
We use the backward Euler method for time discretization of the system \eqref{eq:cont_time_matrix_form}. Let us define a uniform time grid $\omega_t = \{ t^n = n \tau, \ n = 0, 1, \ldots, N_{\text{t}}-1, \ \tau N_{\text{t}} = t_{\text{max}} \}$, where $\tau$ is a time step, $N_{\text{t}}$ is a count of time steps, and $t_{\text{max}}$ is a solution time. Then, we obtain the following fully discrete problem: Find the vector $p^{n} = [p_1^n, p_2^n, ..., p_{N_v^h}^n]$ such that
\begin{equation}\label{eq:matrix_form}
M \frac{p^{n}-p^{n - 1}}{\tau} + A p^{n} = b^{n}, \quad n = 1, 2, \ldots, N_{\text{t}},
\end{equation} 
where $b^{n} = b(p^{n})$ and $p^{n} \approx p(t^n)$.

Note that we can represent this problem in the following form
\begin{equation}\label{eq:matrix_form_2}
p^{n}= (M + \tau A)^{-1} (M p^{n - 1} + \tau b^{n}), \quad n = 1, 2, \ldots, N_{\text{t}}.
\end{equation} 
However, its large dimension makes solving this fine-scale problem computationally challenging. Moreover, the multiscale heterogeneous media can be very complicated and have high contrast; additionally, the construction of coarse mesh may be impractical or impossible. Then,  using standard GMsFEM for model reduction can be impractical. Therefore, in the next section, we present the coarse meshfree GMsFEM for efficient spatial approximation. The reader must keep in mind that the presence of high-contrast coefficients also negatively affects the time stability of the methods, which may restrict the size of the time step, resulting in less efficient and less accurate approximations of the solution at the final time \cite{contreras2023exponential,poveda2024second}. 

\section{Coarse meshfree generalized multiscale finite element method}\label{sec:meshfree_gmsfem}
This section presents the (coarse) meshfree GMsFEM for the flow problem in multiscale heterogeneous media (MFGMsFEM). We follow \cite{nikiforov2023meshfree}. The method was presented as a modification of the standard GMsFEM for the case where the construction of the coarse spaces requires only an overlapping partition into neighborhoods and does not require a coarse mesh 
per se. In this case, the support of the coarse basis functions, or neighborhoods, is not the union of distinguished coarse elements as in the original GMsFEM.   The main advantage of the MFGMsFEM is the capability to consider very complex fine grids and heterogeneous coefficients without needing to construct a suitable coarse grid. As in the standard GMsFEM, the method consists of offline and online stages. The offline stage consists of the following steps:
\begin{enumerate}
\item Generation of a meshfree coarse-scale point cloud;
\item Construction of a multiscale space;
\item Assembling a projection matrix into the multiscale space.
\end{enumerate}
In the online stage, the fine-scale system is converted to a coarse-scale system using the resulting projection matrix. As a result, the problem is solved on a coarse scale with a few degrees of freedom while maintaining high accuracy. See \cite{nikiforov2023meshfree}.
The MFGMsFEM has been successfully applied to several problems; see 
 \cite{djulustan2024meshfree, nikiforov2023meshfree,nikiforov2024meshfree,nikiforov2024meshfree1}. Here, we apply MFGMsFEM to 
 the problem \eqref{eq:fine_sys} for the space discretization and combine it with an exponential integration for the time discretization in such a way that we project all linear systems and function of matrix computation to the coarse scale, see \cite{contreras2023exponential, poveda2024second}.

\subsection{Meshfree coarse scale}\label{subsec:coarse_scale}
In the MFGMsFEM, we use a point cloud instead of a structured coarse grid. Let $\mathcal{S}_H$ be an overlapping partition of the computational domain $\Omega$ to the point cloud in such a way that $\Omega \subset \bigcup_{i=1}^{N_v^H} S_i$, and suppose that each coarse subdomain (or neighborhood) $S_i$ is partitioned into a connected union of fine-grid elements (see Fig. \ref{ris:coarse_example}). Let ${\{ x_i \}}^{N_v^H}_{i=1}$ \corr{be} the coarse-scale points, where $N_v^H$ denotes the number of points. Here, the coarse elements $S_i$ are the supports of basis functions
\begin{equation} \label{eq:def:neigh:juan}
S_i =\bigcup
\left\{ K^h :  \max_{y\in K^h}\ \parallel y- x_i \parallel \le r_i, \quad K^h\in \mathcal{T}^h  \right\},
\end{equation} 
where $r_i$ is the radius of the coarse element $S_i$.

\begin{figure}[!hbt]
\begin{center}
\includegraphics[width=0.4\linewidth]{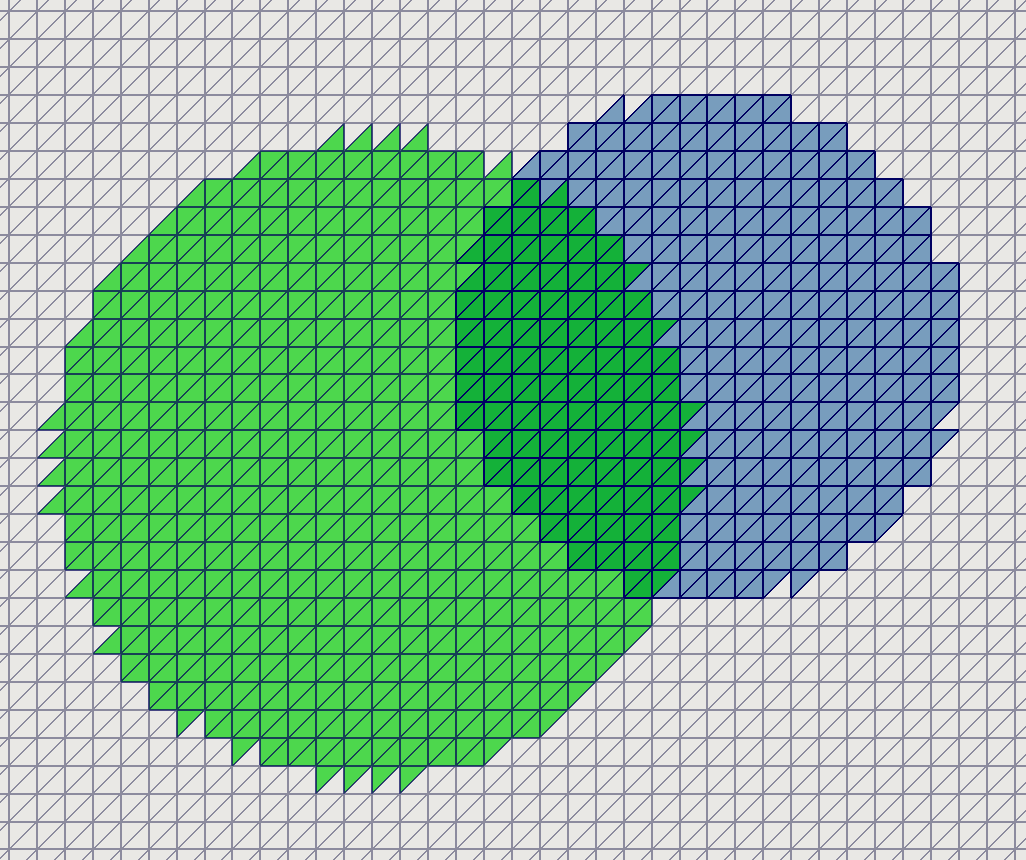}
\end{center}
\caption{Example of some coarse-scale elements $S_i$ (green) and $S_j$ (blue).}
\label{ris:coarse_example}
\end{figure}

We denote the basis functions by $\psi_{i, k}$, which is supported in $S_i$, and the index $k$ represents the numbering of these basis functions. In turn, the solution will be sought as
\begin{equation}\label{eq:ms_sol}
p_{\text{ms}}(x, t) = \sum_{i, k} p_{\text{ms}, ik}(t) \psi_{ik}^{\text{ms}}(x).
\end{equation} 

The main parameters of the MFGMsFEM are the choice of the location of points $\{x_i\}_{i=1}^{N_v^H}$, as well as the sizes of $r_i$ for $\{S_i\}_{i=1}^{N_v^H}$. The version of CVT (centroidal Voronoi tessellations) proposed in \cite{du2002meshfree} is used in this work (see Fig. \ref{ris:pe_pic}). An important aspect of this approach is the choice of the density distribution function $\rho(x)$, which is calculated based on the high-contrast permeability coefficient $\kappa(x)$ and the initial condition $p_0$ using the following problem
\begin{equation}\label{eq:rho_problem}\begin{split}
\beta \div \left( \kappa(x) \grad p \right) + p = p_0, \quad &x \in \Omega \times I,\\
p = 0, \quad &x \in \partial \Omega \times I,
\end{split}\end{equation} 
where $\beta$ is the smoothing parameter, which, in our case, was chosen to be $\beta = 0.01$ by numerical experiments. In this problem \eqref{eq:rho_problem}, we smooth out the heterogeneous initial condition $p_0$, taking into account permeability $\kappa(x)$. These two parameters create computational difficulties in our case. For more details, see \cite{nikiforov2023meshfree}.

For computing the radii $r_i$, it is necessary to consider the fulfillment of the condition of covering the entire region $\Omega \subset \bigcup_{i=1}^{N_v^H} S_i$. The algorithm proposed in \cite{du2002meshfree} is also used. In this algorithm, the parameter $\gamma$ is responsible for the complete coverage of $\Omega$ and the smoothness of the shape functions $W_i(x)$. It is worth noting that it is necessary to use $\gamma > 1$.

\begin{figure}[!hbt]
\begin{center}
\input{point_cloud.tex}
\end{center}
\caption{Point cloud illustration obtained with the 
centroidal Voronoi tessellations method. See \cite{du2002meshfree}.}
\label{ris:pe_pic}
\end{figure}
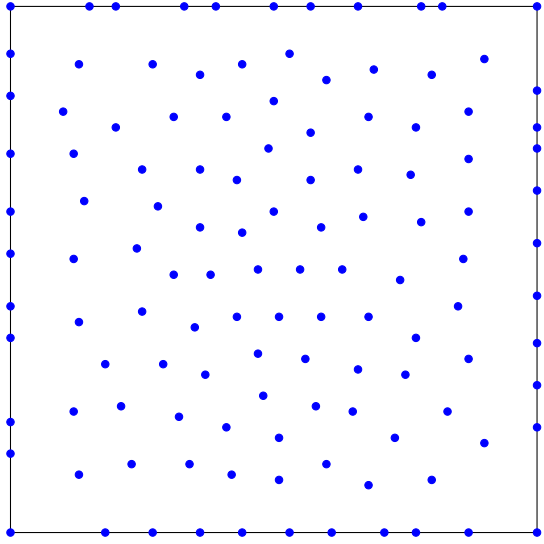

\subsection{Local basis functions}\label{subsec:gmsfem_basis}
We need to solve local spectral problems in each coarse element to construct multiscale basis functions. This step allows us to identify the most dominant modes and capture important fine-scale information. Thus, in each $S_i$, we need to find $\lambda_k^{S_i} \in \mathbb{R}$ and $\psi_k^{S_i} \in H^1(S_i)$ such that
\begin{equation}\label{eq:spectral_problem}
a^{S_i} (\psi_k^{S_i}, v) = \lambda_k^{S_i} m^{S_i} (\psi_k^{S_i}, v), \quad \text{for all } v \in H^1(S_i),
\nonumber
\end{equation}
where $a^{S_i}$ and $m^{S_i}$ are bilinear forms defined as follows
\begin{equation}
\begin{split}
a^{S_i} (\psi_k^{S_i}, v) = \int_{S_i} \kappa(x) \grad \psi_k^{S_i} \cdot \grad v dx, \quad 
m^{S_i} (\psi_k^{S_i}, v) = \int_{S_i} \kappa(x) \psi_k^{S_i} v dx.
\end{split} 
\nonumber
\end{equation}
Next, normalize the eigenvectors such that 
$\int_{S_i}\kappa |\psi_k^{S_i} |^2=1$ and we sort the obtained eigenvalues in ascending order
$
\lambda_1^{S_i}\leq \lambda_2^{S_i}\leq \dots  .
$
Then, we choose $N_{\text{b}}$ eigenfunctions, which correspond to the first $N_{\text{b}}^{S_i}$ eigenvalues to be included in the multiscale space, say $\{ \psi_k^{S_i}\}_{k=1}^{N_{\text{b}}^{S_i}}$. They are also called local spectral basis functions.

We note that, with the local bilinear forms introduced above, in each neighborhood $S_i$, the eigenvalue problem above corresponds to the weak form of the eigenvalue problem
\begin{equation}
\label{eq:strong_eig}
-\mbox{div}(\kappa  \nabla \psi_\ell^{S_i})=\lambda_k^{S_i} \kappa  \psi_\ell^{S_i},
\end{equation}
with homogeneous Neumann
boundary condition on $\partial S_i$.

\subsection{Global formulation}\label{subsec:global_formulation}
In the standard GMsFEM, one multiplies the eigenfunctions by the partition of unity functions to ensure the conformality of basis functions. In the meshfree modification, specific shape functions $W_i$ defined on $S_i$ play their role. We use them to form the initial coarse space
\begin{equation} 
V^{\text{init}}_0 = \text{span} \{ W_i(x):~ 1 \leq i \leq N_v^H \}.
\nonumber
\end{equation}
In our work, we use the following shape functions
\begin{equation}
W_i(x) = \frac{\phi_i(x)}{\sum_{j=1}^N \phi_j(x)},
\nonumber
\end{equation}
where $\phi_i(x)$ are kernel functions defined as the following cubic splines
\begin{equation}
\phi_i(r) = 2 \left\{
\begin{array}{lr}
2/3 + 4 \ (r-1) \ r^2,& r \le 0.5,\\
4/3 \ (1-r)^3,& 0.5 \le r \le 1,\\
0,& 1 \le r.
\end{array}
\right.
\nonumber
\end{equation}
Here, $r = \frac{||x - x_i||}{r_i}$ is the normalized distance. We multiply the eigenfunctions $\psi_k^{S_i}$ by the shape functions $W_i$ to obtain multiscale basis functions; that is, we introduce
\begin{equation}
\psi_{ik}^{\text{ms}} = W_i \psi_k^{S_i}, \ 1 \le i \le N_v^H, \ 1 \le k \le N_{\text{b}},
\nonumber
\end{equation}
where $k$ is an index of the eigenfunction. Figure \ref{ris:basis1} illustrates a shape function $W_i$, an eigenfunction $\psi_k^{S_i}$, and the resulting multiscale basis function.

Note that one can use several multiscale basis functions per coarse element. Moreover, the accuracy of the multiscale method is controlled by the number of basis functions. The resulting multiscale space is defined as follows
\begin{equation}
\begin{split}
V_{0} = \text{span} \{ \psi_{ik}^{\text{ms}}: 1 \le i \le N_v^H, \ 1 \le k \le N_{\text{b}}^{S_i} \}.
\end{split}
\nonumber
\end{equation}

We search for a solution in the multiscale space that contains significantly fewer degrees of freedom than the standard fine-scale one. Moreover, we do not need a coarse grid to construct the multiscale space. Therefore, we have the following problem: Find $p_{\text{ms}}(t) \in V_0$ such that
\begin{equation}\label{eq:var_form_ms}
m\left(\frac{\partial p_{\text{ms}}}{\partial t}, v\right) + a(p_{\text{ms}}, v) = F(p_{\text{ms}}; v), \quad \text{for all } v \in \widehat{V}_0,
\end{equation}
where the forms $m$, $a$, and $F$ are defined in \eqref{eq:fine_forms}.

Note that we can represent this problem in matrix form. Let us denote by $N_{\text{c}} = N_v^H N_{\text{b}}$ the number of degrees of freedom of the multiscale problem. Next, we renumber the eigenfunctions from $1$ to $N_{\text{c}}$. Then, we introduce the projection matrix into the multiscale space $R_0$ as follows
\begin{equation}\label{eq:projection_matrix}
R_0 = [\psi_{1}^{\text{ms}}, \psi_{2}^{\text{ms}}, ..., \psi_{N_{\text{c}}}^{\text{ms}}]^T,
\end{equation}
where $\psi_{k}^{\text{ms}}$ are the fine-scale coordinate representations of the corresponding eigenfunctions.

Finally, we have the following problem in a continuous time matrix form: Find the vector $p_{\text{ms}}(t) = [p_{\text{ms}, 1}(t), ..., p_{\text{ms}, N_{\text{c}}}(t)]^T$ such that
\begin{equation}\label{eq:coarse_problem} 
M_{0} \frac{d p_{\text{ms}}}{d t} + A_{0} p_{\text{ms}} = b_0(p_{\text{ms}}),
\end{equation}
where $M_0$, $A_0$, and $b_0$ are obtained using the corresponding fine-scale matrices and vectors defined in \eqref{eq:fine_matrices_and_vectors} and the projection matrix $R_0$ as follows
\begin{equation}
M_0 = R_0 M R_0^T, \quad A_0 = R_0 A R_0^T, \quad b_0 = R_0 b.
\nonumber
\end{equation}

After solving the system \eqref{eq:coarse_problem}, we can go from a coarse-scale solution to a fine-scale solution also using the projection matrix $R_0$ and the solution $p_\text{ms}$
\begin{equation}
p \approx R_0^T p_{\text{ms}}, 
\nonumber
\end{equation}
where $p$ is the fine-scale coordinate vector.

\begin{figure}[!hbt]
\begin{center}
\begin{minipage}[h]{0.3\linewidth}
\center{\includegraphics[width=\linewidth]{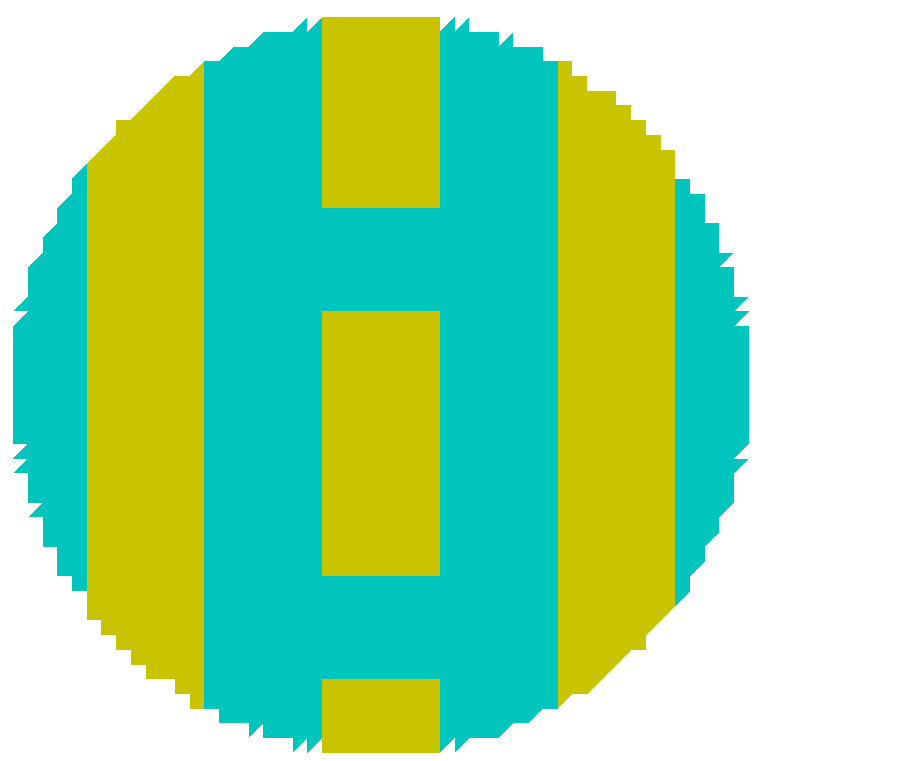}}
\end{minipage}
\begin{minipage}[h]{0.3\linewidth}
\center{\includegraphics[width=\linewidth]{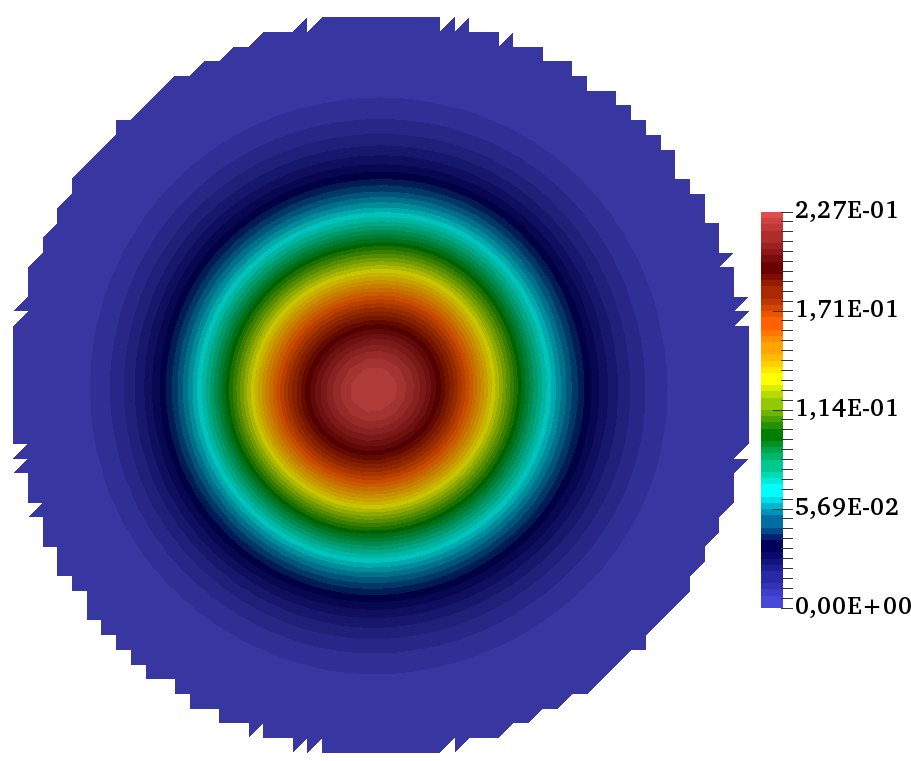}}
\end{minipage}\\
\begin{minipage}[h]{0.3\linewidth}
\center{\includegraphics[width=\linewidth]{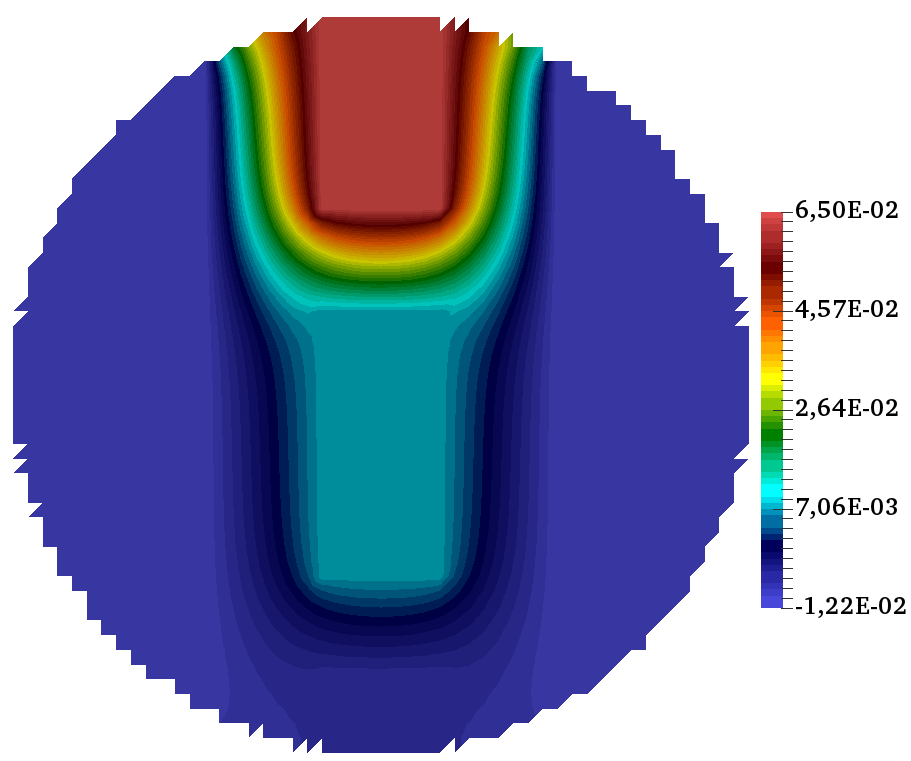}}
\end{minipage}
\begin{minipage}[h]{0.3\linewidth}
\center{\includegraphics[width=\linewidth]{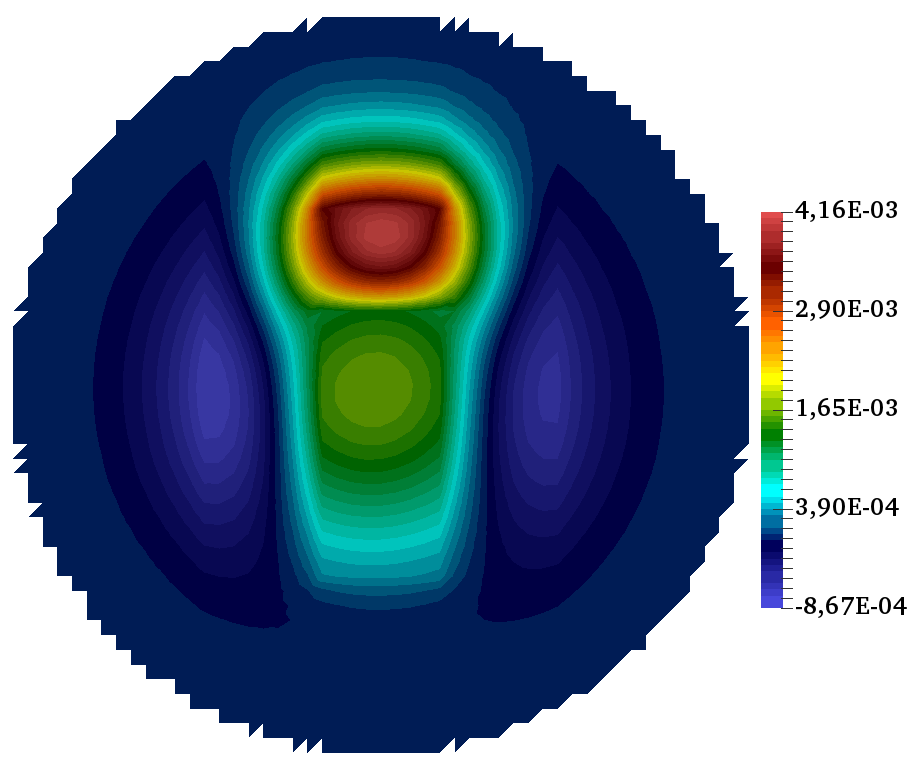}}
\end{minipage}
\end{center}
\caption{Illustration of high-contrast coefficient $\kappa(x)$ and multiscale functions at $\gamma = 4$. Top-left: $\kappa(x)$ in $S_i$. Top-right: shape function $W_i(x)$. Bottom-left: eigenfunction $\psi_k^{S_i}$. Bottom-right: multiscale basis function $\psi_{ik}^{\text{ms}}$.}
\label{ris:basis1}
\end{figure}

\subsection{The MFGMsFEM coarse interpolation}
In this section, we design and analyze a suitable coarse interpolant.  We follow the ideas in 
\cite{abreu2019convergence,egw10}.
Note that $\{W_i\}$ form a partition of unity subordinated to the covering $\{S_i\}_{i=1}^{N_v^H}$. Given a $u\in V$, 
we recall that 
\[u|_{S_i} = \sum_{k=1}^{N^{S_i}}\left(\int_{S_i}\kappa(x) u \psi_k^{S_i} dx\right) \psi_k^{S_i},
\]
where $N^{S_i}$ is the total number of vertices in $S_i$.
As in \cite{ge09_1,ge09_1reduceddim,abreu2019convergence}, due to eigenvector orthogonality properties, we have, 
\[
\int_{S_i}\kappa(x) |u|^2 dx=\sum_{k=1}^{N^{S_i}}\left(\int_{S_i}\kappa(x) u\psi_k^{S_i}dx \right)^2,
\]
and 
\[
\int_{S_i}\kappa(x)|\nabla u|^2dx=\sum_{k=1}^{N^{S_i}}\lambda_k^{S_i}\left(\int_{S_i}\kappa(x) u \psi_k^{S_i}dx \right)^2.
\]
Define coarse interpolation as the combination, weighted by the partition of unity, of projections on the space generated by the selected $N_b^{S_i}$ local eigenvectors. That is, we define, 
\[
I_{0}u =\sum_{i=1}^{N_v^H} \left(\sum_{k=1}^{N_b^{S_i}}
\left(\int_{S_i}\kappa(x) u\psi_k^{S_i}dx \right)\psi_k^{S_i}\right)W_i =\sum_{i=1}^{N_v^H}(I^{S_i} u)W_i,
\]
where 
\[
I^{S_i}u=\sum_{k=1}^{N^{S_i}_b}\left(\int_{S_i}\kappa(x)u\psi_k^{S_i}dx\right)\psi_k^{S_i}
\]
is the projection of $u|_{S_i}$ onto the space of the first $N_b^{S_i}$ eigenvectors. Then, 
\begin{equation}\label{eq:decompo}
u-I_{0}u =\sum_{i=1}^{N_v^H} \left( u-\sum_{k=1}^{N_b^{S_{i}}}
\left(\int_{S_{i}}\kappa(x) u\psi_k^{S_{i}}dx\right)\psi_k^{S_{i}}\right)W_{i} = \sum_{i=1}^{N_v^H} \left( \sum_{k=N_b^{S_{i}}+1}^{N^{S_{i}}}\left(\int_{S_{i}}\kappa(x)u\psi_k^{S_{i}}dx\right)\psi_k^{S_i}\right)W_{i}.
\end{equation}
Again, as in \cite{ge09_1,ge09_1reduceddim,abreu2019convergence}
and due to eigenvector orthogonality properties, we have 
\begin{equation}\label{eq:truncation}
\int_{S_{i}}\kappa(x)|u-I^{S_{i}} u|^2dx\leq\frac{1}{\lambda_{N_b^{S_{i}}+1}^{S_{i}}}\int_{S_{i}}\kappa(x)|\nabla u|^2dx.
\end{equation}
and 
\begin{equation}\label{eq:truncationgrad}
\int_{S_{i}} \kappa(x)|\nabla u-\nabla I^{S_{i}} u|^2dx\leq\frac{1}{\lambda_{N_b^{S_{i}}+1}^{S_{i}}}\sum_{k={N^{S_{i}}_{b}}+1}^{N^{S_{i}}}\left(\lambda_k^{S_{i}}\right)^2\left(\int_{S_{i}}\kappa u
\psi_k^{S_{i}}dx\right)^{2}\leq\frac{1}{\lambda_{N_{b}^{S_{i}}+1}^{S_{i}}} D_{i}(u),
\end{equation}
where 
\[
D_{i}(u)=\sum_{k=N_b^{S_{i}}+1}^{N^{S_{i}}}\left(\lambda_k^{S_{i}}\right)^2\left(\int_{S_{i}}\kappa u\psi_k^{S_i}dx\right)^2.
\]
We will assume that $D_i(u)$ is bounded, which, as explained in \cite{abreu2019convergence}, is basically a regularity assumption for the solution. For more explanations about such a quantity and how to relate it to local and global norms, see \cite{abreu2019convergence}.
In this case, we also have,
\begin{equation}\label{eq:truncationD}
\int_{S_{i}}\kappa(x)|u-I^{S_i} u|^2dx\leq \frac{1}{(\lambda_{N_b^{S_{i}}+1}^{S_{i}})^2}D_i(u).
\end{equation}
Following \cite{tw}, let 
\begin{equation} \label{eq:def:cov}
C_{\rm{ov}} = \max_{x\in \Omega} \# \{ i : x\in S_{i}\}.
\end{equation}
That is, $C_{\rm{ov}}$ is a bound for the maximum number of neighborhoods with a non-empty intersection of the overlapping partition $\{  S_{i}\}$. As in \cite{tw} 
(see Assumptions 2.3, 2.4, and 3.2 and Equation (2.16)  in \cite{tw} or Appendix A), we can write
\begin{align}
\int_\Omega \kappa(x)| u- I_{0}u|^2dx&\leq C_{\rm{ov}}\sum_{i=1}^{N_v^H}  \int_{S_{i}} \kappa(x) W_i^2| u- I^{S_i}u|^2dx \label{StrengthenedCS}\\
&\leq C_{\rm{ov}}\sum_{i=1}^{N_v^H}\int_{S_i}\kappa(x)| u- I^{S_i}u|^2dx \nonumber  \\
&\leq C_{\rm{ov}}\sum_{i=1}^{N_v^H}\frac{1}{\lambda_{N_b^{\tau_i}+1}^{S_{i}}}
\int_{S_{i}}\kappa(x)|\nabla u|^2dx \nonumber\\
&\leq C_{\rm{ov}} \max_{1\leq i\leq N_v^H}\frac{1}{\lambda_{N_b^{S_i}+1}^{S_i}}\sum_{i=1}^{N_v^H}\int_{S_{i}}\kappa(x)|\nabla u|^2dx \nonumber\\
&\leq C_{\rm{ov}}^2\max_{1\leq i\leq N_v^H}\frac{1}{\lambda_{N_b^{S_i}+1}^{S_i}}\int_{\Omega}\kappa(x)|\nabla u|^2dx. \nonumber
\end{align}
Analogously, 
\begin{equation}\label{eq:grad-interp}
\begin{split}
\int_{\Omega}\kappa(x)| \nabla u- \nabla I_{0}u|^2dx&\leq C_{\rm{ov}}\sum_{i=1}^{N_v^H}\int_{S_{i}} \kappa(x)|\nabla (W_i( u- I^{S_{i}}u))|^2dx\\
&\leq 2C_{\rm{ov}}\left(\sum_{i=1}^{N_v^H}\int_{S_{i}} \kappa(x) W_{i}^2|\nabla ( u- I^{S_{i}}u)|^2dx + \int_{S_{i}}\kappa(x)|\nabla W_{i}|^2| u- I^{S_{i}}u|^2dx\right)\\
&\leq 2C_{\rm{ov}}\left(\sum_{i=1}^{N_v^H} \int_{S_{i}}\kappa(x)|\nabla ( u- I^{S_{i}}u)|^2dx + \max_{x\in S_{i}} |\nabla W_{i}(x)|^2\int_{S_{i}} \kappa(x) | u- I^{S_{i}}u|^2dx\right)\\
&\leq 2C_{\rm{ov}}  \max_{1\leq i\leq N_v^H}\left\{ \frac{1}{\lambda^{S_{i}}_{N_b^{S_{i}}+1}},\frac{\max_{x\in S_{i}} |\nabla W_{i}(x)|^2}{\left(\lambda_{N_b^{S_{i}}+1}\right)^2}\right\}\sum_{i=1}^{N_v^H} D_{i}(u).
\end{split}
\end{equation}
We recall that $\lambda_{N_b^{S_{i}}+1}$ decreases and it scale as $H^2$ in two dimensions. Also, recall  that $|\nabla W_{i}|$ scales as $\delta^{-1}$, the inverse of overlapping parameter $\delta$,  where, see \cite{tw}, 
\[
\delta = \min_{1\leq i\leq N_v^H} \delta_{i},
\]
where 
\[
\delta_{i}=\min\left\{\|x-y\| : x\in S_{i}\setminus\bigcup_{\ell\neq i} S_{\ell},\quad y\in \partial S_{i}\right\}.
\]
\section{Temporal discretizations combined with MFGMsFEM spatial approximation}\label{sec:gmsfem_time_discretization}

This section describes temporal discretization methods for the ODE system equation \eqref{eq:cont_time_matrix_form}. In particular, we consider how the spatial approximation of the MFGMsFEM can be combined with finite difference and exponential integration for time discretization. The former approach is the most common but unstable in high-contrast multiscale media requiring minimal time stepping (\cite{poveda2024second, contreras2023exponential}). As a promising alternative, the second approach allows us to take full advantage of the spatial model reduction capabilities of the MFGMsFEM and achieve better stability properties.

\subsection{MFGMsFEM combined with finite difference (MFGMsFEM-FD)}\label{subsec:gmsfem_fd}

Various methods in the finite difference approach can be used for time discretization of the ODE system in \eqref{eq:cont_time_matrix_form}. This work considers the backward Euler method, which is most often used with GMsFEM. Similar to the fine-grid approximation (Section \ref{subsec:fine_grid_fd}), we use a uniform time grid $\omega_t = \{ t^n = n \tau, \ n = 0, 1, \ldots, N_{\text{t}}-1, \ \tau N_{\text{t}} = t_{\text{max}} \}$, where $\tau$ is a time step and $N_{\text{t}}$ is the number of time steps. Then, by replacing the time derivative with the finite difference, we obtain the following problem: Find $p_{\text{ms}}^{n} = [p_{\text{ms},1}^{n}, ..., p_{\text{ms},N_{\text{c}}}^{n}]$ such that

\begin{equation}\label{eq:coarse_fd_problem} 
M_{0} \frac{p_{\text{ms}}^{n} - p_{\text{ms}}^{n - 1}}{\tau} + A_{0} p_{\text{ms}}^{n} = b_0^{n},
\end{equation}
where $b_{0}^{n} = b_0(p_{{\text{ms}}}^{n})$ and $p_{\text{ms}}^{n} \approx p_{\text{ms}}(t^n)$.

Note that one can represent this discrete problem in the following form for the fine-grid solution $p^{n}_{\text{ms}}$
\begin{equation}\label{eq:coarse_fd_problem_2}
p^{n}_{\text{ms}}= R_0^T (M_0 + \tau A_0)^{-1} (M_0 p_{\text{ms}}^{n - 1} + \tau b_0^{n}), \quad n = 1, 2, \ldots, N_{\text{t}}.
\end{equation} 
or
\begin{equation}\label{eq:coarse_fd_problem_3}
p^{n}_{\text{ms}}= R_0^T (M_0 + \tau A_0)^{-1} R_0 (M p^{n - 1} + \tau b^{n}), \quad n = 1, 2, \ldots, N_{\text{t}},
\end{equation} 
which makes sense since the real bottleneck of the fine-scale problem is computing the solution of the linear system with matrix $M + \tau A$. However, as the results for the standard GMsFEM showed in \cite{contreras2023exponential}, the finite difference temporal approximation can be unstable and requires small time steps to approximate the fine-scale solution accurately.

\subsection{MFGMsFEM  combined with exponential integrator (MFGMsFEM-EI)}\label{subsec:gmsfem_ei}

We propose exponential integration as an alternative time approximation approach. Based on \cite{hochbruck2010exponential, contreras2023exponential,poveda2024second}, let us give a brief outlook of this method. Let us consider the fine-scale  system in \eqref{eq:cont_time_matrix_form} and transform it into the following form
\begin{equation}\label{eq:cont_time_matrix_form_2}
\begin{split}
\frac{d p}{d t} + N p &= F(p),\\
p(0) &= p_0,
\end{split}
\end{equation}
where $N = M^{-1} A$, $F = M^{-1} b(p)$, and $p_0 = M^{-1} \hat{p}$.

This matrix ODE system has an exact solution, expressed using matrix exponential as follows
\begin{equation}\label{eq:variation_of_constants}
p (t^n) = e^{-\tau N} p(t^{n - 1}) + \int_{0}^{\tau} e^{(l - \tau) N} F(p(t^{n - 1} + l)) dl.
\end{equation}

This solution representation is also called the variation of constants. Next, we need to obtain an approximation of the nonlinear term by the interpolation polynomial on quadrature nodes. We approximate the integral using exponential quadrature rule with $s$ quadrature points $c_i \in [0, 1]$ in the following way \cite{hochbruck2010exponential}
\begin{equation}\label{eq:exp_quadrature_rule}
p^n = e^{-\tau N} p^{n - 1} + \tau \sum_{i = 1}^{s} b_i (-\tau N) F_i,
\end{equation}
where $p^n \approx p(t^n)$, $F_i = F(p(t^{n - 1} + c_i \tau))$, and the coefficients $b_i$ are following
\begin{equation}
\nonumber
b_i (-\tau N) = \int_{0}^{1} e^{-\tau (1 - \theta) N} l_i(\theta) d\theta.
\end{equation}

Here, $l_i$ denotes the Lagrange interpolation polynomials. Note that one can represent the coefficients $b_i (z)$ as linear combinations of $\varphi$-functions
\begin{equation}\label{eq:varphi_definition}
\begin{split}
\varphi_0 (z) = e^z, \quad
\varphi_{p + 1} (z) = \frac{1}{z} \left( \varphi_p(z) - \frac{1}{p!} \right).
\end{split}
\end{equation}

Moreover, one can obtain the following recurrence relation for $\varphi$-functions
\begin{equation}\label{eq:varphi_definition_2}
\varphi_p (z) = \int_0^1 e^{(1 - \theta) z} \frac{\theta^{p - 1}}{(p - 1)!} d\theta, \quad p \geq 1.
\end{equation}

For simplicity, let us consider a scheme with only one quadrature node, i.e., $s = 1$ and $c_1 = 0$. Therefore, we have a constant interpolation polynomial and get $b_1(-\tau N) = \int_0^1 e^{-\tau (1 -\theta) N} d \theta = \varphi_1(-\tau N)$, according to \eqref{eq:varphi_definition_2}. Then, using $e^z = z \varphi_1(z) + 1$ and $F_1 \approx F^{n - 1} = F(p^{n - 1})$ in \eqref{eq:exp_quadrature_rule}, we obtain
\begin{equation}\label{eq:exp_euler_scheme}
p^n = p^{n - 1} + \tau \varphi_1 (-\tau N) (F^{n - 1} - N p^{n - 1}),
\end{equation}
which is called the exponential Euler method.

Note that, using the symmetric and positive definite properties of $A$ and $M$, one can show \cite{contreras2023exponential}
\begin{equation}
\nonumber
-\tau N = Q D Q^T M
\end{equation}
and
\begin{equation}
\nonumber
\varphi_p(-\tau N) = Q \varphi_p (D) Q^T M,
\end{equation}
where $Q$ is a matrix, which columns are eigenvectors of $-\tau N q = \lambda q$ or $-\tau A q = \lambda M q$.

In this way, we obtain the following form of the exponential integration scheme
\begin{equation}\label{eq:exp_euler_scheme_2}
p^n = p^{n - 1} + \tau Q \varphi_1(D) Q^T (M F^{n - 1} - A p^{n - 1}).
\end{equation}

\corr{Note that the exponential Euler method is impractical for multiscale heterogeneous problems due to the high computational cost of computing matrix functions. However, the MFGMsFEM allows us to significantly reduce the problem dimensionality, making it possible to use the exponential integration even for problems in complex heterogeneous media with high contrast.}

We propose the approximation of the $\varphi$-functions
\begin{equation}
\nonumber
\varphi_p(-\tau N) \approx R_0^T \varphi_p (-\tau N_0) R_0,
\end{equation}
where $N_0 = M_0^{-1} A_0$. We also have the following eigenvalue decomposition
\begin{equation}\label{eq:mfgsmfem_ei_eigenvalue}
-\tau N_0 = Q_0 D_0 Q_0^T M_0,
\end{equation}
where $Q_0$ are the matrix, which columns are the eigenvectors of $-\tau A_0 q_0 = \lambda M_0 q_0$.

Therefore, the MFGMsFEM-EI scheme is given by
\begin{equation}\label{eq:mgsfem_ei}
p^n = p^{n - 1} + \tau R_0^T \varphi_1 (-\tau N_0) R_0 (F^{n - 1} - N p^{n - 1}).
\end{equation}

There are various ways to calculate matrix functions. One of the popular ways is Pad\'e approximations. One can also apply the eigenvalue decomposition approach \eqref{eq:mfgsmfem_ei_eigenvalue}. In the latter case, the exponential integration will have the following form
\begin{equation}\label{eq:mgmsfem_ei_2}
p^n = p^{n - 1} + \tau R_0^T Q_0 \varphi_1 (D_0) Q_0^T R_0 (MF^{n - 1} - A p^{n - 1}).
\end{equation}

Therefore, the algorithm of the MFGMsFEM-EI consists of the following steps:
\begin{enumerate}
\item Compute the fine-scale residual $r_h^{n - 1} = M F^{n - 1} - A p^{n - 1}$;
\item Upscale the residual $r_H^{n - 1} = R_0 r_h^{n - 1}$.
\item Compute the coarse-scale function of matrix $\delta_H^{n - 1} = Q_0 \varphi_1(D_0) Q_0^T r_H^{n - 1}$.
\item Find the fine-scale solution $p^{n} = p^{n - 1} + \tau R_0^T \delta_H^{n - 1}$.
\end{enumerate}

It should be noted that since we consider time evolution on a coarse scale, it may be helpful to use the orthogonal projection $\hat{p}_0 = R_0^T M_0^{-1} R_0 M p_0$ as the initial condition.

\section{Convergence analysis}\label{sec:convergence}
In this section, we focus on error estimates of fully discrete solutions given by the MFGMsFEM-EI for solving the semilinear parabolic problem with Dirichlet boundary conditions. First, we shall introduce some notations and usual approximation results to estimate the error bound. Let us define the following norms for our analysis
\[
\|p\|_{a_{i}}^{2}:=\int_{S_{i}}k|\grad p|^{2}dx,\quad\|p\|_{m_{i}}^{2}:=\int_{S_{i}}k(x)|p|^{2}dx.
\]
In addition, we regard some regularity and growth conditions for solution $p$ and reaction term $f$ to carry out the error analysis of our proposed method, see \cite{thomee2006galerkin}.

\begin{asp}
\label{asp:01}
The exact solution $p(t)$ satisfies the following regularity conditions:
\begin{subequations}
\begin{align}
\sup_{0\leq t\leq T}\|p(\cdot,t)\|_{H^{2}}&\preceq 1,\label{eq:asp3-1}\\
\sup_{0\leq t\leq T}\|\partial_{t} p(\cdot,t)\|_{L^{\infty}}&\preceq 1,\label{eq:asp3-2}\\
\sup_{0\leq t\leq T}\|\partial^{2}_{tt} p(\cdot,t)\|_{L^{\infty}}&\preceq 1\label{eq:asp3-3},
\end{align}
\end{subequations}
where the hidden constants may depend on $T$.
\end{asp}

\begin{asp}
\label{asp:02}
The function $f(p)$ grows mildly regarding $p$, i.e., there exists a number $\ell>0$ for $d=1,2$ or $\ell\in(0,2]$ for $d=3$ such that
\begin{equation}
|f'(p)|\preceq 1 + \|p\|^{\ell},\quad\mbox{for all } p\in\mathbb{R}.
\end{equation}
\end{asp}
Finally, we have the following result on the locally-Lipschitz continuity of the reaction term $f$.
\begin{lem}
\label{lem:local-f}
Suppose that the function $f$ satisfies Assumption \ref{asp:02}, and the exact solution $p(\cdot,t)$ satisfies Assumption \ref{asp:01}. Then, $f$ is locally-Lipschitz continuous in a strip along the exact solution $p(\cdot,t)$,~i.e., for any given positive constant $C$,
\begin{equation}
\label{eq:lem2-1}
\|f(v)-f(w)\|_{L^{2}}\preceq \|v-w\|_{H^{1}},
\end{equation}
for any $t\in[0,T]$ and $v,w\in V$ satisfying $\max\{\|(v-p(t)\|_{H^{1}},\|w-p(t)\|_{H^{1}}\}\leq C$, where the hidden constant in \eqref{eq:lem2-1} may depend on $C$.
\end{lem}


In the following, we derive the error estimate between the exact solution $p(\cdot,t_{n})$ and the fully discrete multiscale solution $p^{n}_{\text{ms}}$ using the framework of \cite{abreu2019convergence,chung2018constraint}. For simplicity of presentation, we shall assume that the partition is uniform in $I$ with time step $\tau$. Let $p_{\text{ms}}(t)$ be the multiscale solution of the semi-discrete problem \eqref{eq:var_form_ms}, and $p^{n}_{\text{ms}}$ the multiscale fully discrete solution produced by the exponential integrator multiscale finite element method \eqref{eq:mgsfem_ei}.

Let $\hat{p}\in V_{0}$ be the Ritz projection of the solution $p$ that satisfies
\begin{equation}
\label{eq:proj_p}
a((p-\hat{p})(t),v)=0,\quad\mbox{for each }v\in V_{0},\mbox{and } t\in I.
\end{equation}
The following lemma gives the error estimate of $\hat{p}(t)$ for the semi-parabolic problem.

\begin{lem}
\label{lem:esti-proj}
Let $p$ be the solution of \eqref{eq:fine_sys}. For each $t\in[0,T]$, we define the Ritz projection $\hat{p}\in V_{0}$ satisfies \eqref{eq:proj_p}. Then,
\begin{subequations}
\begin{align}
\|(p-\hat{p})(t)\|_{H^{1}}&\leq C C_{\rm{ov}}^{1/2}\kappa_{\min}^{-1/2}\max_{1\leq i\leq N_v^H}\left(\lambda_{N_b^{S_{i}}+1}^{S_{i}}\right)^{-1/2},\label{eq:lemE1-01}\\
\|(p-\hat{p})(t)\|_{L^{2}}& \leq C C_{\rm{ov}}\kappa_{\min}^{-1}\max_{1\leq i\leq N_v^H}\left(\lambda_{N_b^{S_{i}}+1}^{S_{i}}\right)^{-1},\label{eq:lemE1-02}
\end{align}
\end{subequations}
where $C:=C\left(p,\tfrac{\partial p}{\partial t},f\right)$ is a constant dependent on $p,\tfrac{\partial p}{\partial t}$ and $f$.
\end{lem}
\begin{proof}
Note that $p\in V$ satisfies
\[
a(p,v)=F(p;v)-m\left(\frac{\partial p}{\partial t} ,v\right),\quad\mbox{for each  } v\in \widehat{V},\quad\mbox{for all } t\in I.
\]
Thus, let $\widehat{p}(t)$ be the Ritz projection of $p(t)$ in $V_{0}$ fulfills \eqref{eq:proj_p} for $t\in I$, then 
\[
a(p,v)=a(\hat{p},v)=F(p;v)-m\left(\frac{\partial p}{\partial t},v\right),\quad\mbox{for each }v\in V_{0},\quad\mbox{for all } t\in I.
\]
Putting $v=p-\hat{p}$ in the expression above, we obtain that
\begin{align*}
\|p-\hat{p}\|_{H^{1}}^{2}=a(p-\hat{p},p-\hat{p})=a(p,p-\hat{p})&\leq \left\|F(p;p-\hat{p})-m\left(\frac{\partial p}{\partial t},p-\hat{p}\right)\right\|_{L^{2}}\\
&\leq \kappa^{-1/2}_{\min}\left\|\left(F(p)-\frac{\partial p}{\partial t}\right)\right\|_{L^2}\|p-\hat{p}\|_{m}.
\end{align*}

Notice that, by using the orthogonality of the eigenfunctions $\psi^{S_{i}}_{k}$ of \eqref{eq:strong_eig} and following \cite{abreu2019convergence}, we arrive at
\[
\|p-\hat{p}\|_{m}^{2}\leq C_{\rm{ov}}\left(\lambda_{N_b^{S_{i}}+1}^{S_{i}}\right)^{-1}\|p-\hat{p}\|_{H^{1}}^{2}.
\]
Thus, by gathering the two expressions above, we have that
\[
\|p-\hat{p}\|_{H^{1}}\leq C_{\rm{ov}}^{1/2}\kappa_{\min}^{-1/2}\max_{1\leq i\leq N_v^H}\left(\lambda_{N_b^{S_{i}}+1}^{S_{i}}\right)^{-1/2}\left\|\left(F(p)-\frac{\partial p}{\partial t}\right)\right\|_{L^{2}}.
\]
By using $f(0)=0$, Lemma~\ref{lem:local-f}, and that $p$ satisfies the Assumption \ref{asp:01}, we obtain that $\|f(p)\|_{L^{2}}=\|f(p)-f(0)\|_{L^{2}}\preceq1$. Now, using \eqref{eq:asp3-2} and \eqref{eq:grad-interp}, we find that
\[
\|p-\hat{p}\|_{H^{1}}\leq C C_{\rm{ov}}^{1/2}\kappa_{\min}^{-1/2}\max_{1\leq i\leq N_v^H}\left(\lambda_{N_b^{S_{i}}+1}^{S_{i}}\right)^{-1/2},
\]
where $C=C(p,\tfrac{\partial p}{\partial t},f)$ is a constant depending on $p,\tfrac{\partial p}{\partial t}$ and $f$. Then, we arrive at \eqref{eq:lemE1-01}.  Now, using the duality argument for \eqref{eq:lemE1-02}. For $t\in I$, we define $w\in V_{0}$ such that
\[
a(w,v)=(p-\hat{p},v),\quad\mbox{for each } v\in V_{0},
\]
and define $\hat{w}$ as the Ritz projection of $w$ in the space $V_{0}$, that is,
\[
a(\hat{w},v)=(p-\hat{p},v),\quad\mbox{for each } v\in V_{0}.
\]
By using \eqref{eq:lemE1-01}, we obtain 
\begin{align*}
\|p-\hat{p}\|_{L^{2}}^{2}=a(w,p-\hat{p})&=a(w-\hat{w},p-\hat{p})\\
&\leq \|w-\hat{w}\|_{H^{1}}\|p-\hat{p}\|_{H^{1}}\\ 
&\leq C\left(C_{\rm{ov}}^{1/2}\kappa_{\min}^{-1/2}\max_{1\leq i\leq N_v^H}\left(\lambda_{N_b^{S_{i}}+1}^{S_{i}}\right)^{-1/2}\|p-\hat{p}\|_{L^{2}}\right)\\
&\quad\times\left(C_{\rm{ov}}^{1/2}\kappa_{\min}^{-1/2}\max_{1\leq i\leq N_v^H}\left(\lambda_{N_b^{S_{i}}+1}^{S_{i}}\right)^{-1/2}\right)\\
&\leq CC_{\rm{ov}}\kappa_{\min}^{-1}\max_{1\leq i\leq N_v^H}\left(\lambda_{N_b^{S_{i}}+1}^{S_{i}}\right)^{-1}\|p-\hat{p}\|_{L^{2}}.
\end{align*}
Hence, dividing by $\|p-\hat{p}\|_{L^{2}}$, we obtain \eqref{eq:lemE1-02}. This completes the proof.
\end{proof}

\begin{thm}
\label{thm:estiH1}
Let $p$ be the solution of \eqref{eq:fine_sys} and satisfy Assumption \ref{asp:01}. Assume that the function $f$ fulfills Assumption~\ref{asp:02} and Lemma~\ref{lem:local-f}. We define the elliptic projection $\hat{p}\in V_{0}$ satisfies \eqref{eq:proj_p}. Then, the multiscale solution $p_{\rm{ms}}(t)$ with initial condition $p_{\rm{ms}}(0)$ exists for $t\in[0,T]$ and holds
\begin{subequations}
\begin{align}
\left\|(p-p_{\rm{ms}})(t)\right\|_{H^{1}}&\leq \|(p-p_{\rm{ms}})(0)\|_{H^{1}}+C\left\{C_{\rm{ov}}\kappa_{\min}^{-1}\max_{1\leq i\leq N_v^H}\left(\lambda_{N_b^{S_{i}}+1}^{S_{i}}\right)^{-1}\right.\nonumber\\
&\quad\left.+C_{\rm{ov}}^{1/2}\kappa_{\min}^{-1/2}\max_{1\leq i\leq N_v^H}\left(\lambda_{N_b^{S_{i}}+1}^{S_{i}}\right)^{-1/2}\right\},\label{eq:thmH1main}\\
\left\|(p-p_{\rm{ms}})(t)\right\|_{L^{2}}&\leq \|(p-p_{\rm{ms}})(0)\|_{L^{2}}+CC_{\rm{ov}}\kappa_{\min}^{-1}\max_{1\leq i\leq N_v^H}\left(\lambda_{N_b^{S_{i}}+1}^{S_{i}}\right)^{-1},\label{eq:thmL2main}
\end{align}
\end{subequations}
where constant $C>0$ depends on $T$.
\end{thm}
\begin{proof}
We split the error 
\[
p(t)-p_{\text{ms}}(t)=(p(t)-\hat{p}(t))+(\hat{p}(t)-p_{\text{ms}}(t)):=\rho(t) +\theta(t),\quad \text{for each }t\in[0,T],
\]
where $\hat{p}$ is the elliptic projection in the space $V_{0}$ of the solution $p$. From Lemma~\ref{lem:esti-proj} we immediately get 
\[
\|\rho(t)\|_{H^{1}}\leq C_{\rm{ov}}^{1/2}\kappa_{\min}^{-1/2}\max_{1\leq i\leq N_v^H}\left(\lambda_{N_b^{S_{i}}+1}^{S_{i}}\right)^{-1/2}.
\]
For $\theta$, in virtue of \eqref{eq:var_form_ms} and \eqref{eq:proj_p}, we obtain 
\begin{equation}\label{eq:th-h1}
\begin{split}
m\left(\frac{\partial\theta}{\partial t},v\right)+a(\theta,v)&=m\left(\frac{\partial\hat{p}}{\partial t},v\right)-F(p_{\rm{ms}};v)+a(\hat{p},v)\\
& =  m\left(\frac{\partial(\hat{p}-p)}{\partial t},v\right)+ m\left(\frac{\partial p}{\partial t},v\right)+a(p,v)-F(p_{\rm{ms}};v)\\
& = F(p;v)-F(p_{\rm{ms}};v)-m\left(\frac{\partial\rho}{\partial t},v\right),
\end{split}
\end{equation}
for all $v\in\widehat{V}_{0}$. Taking $v=2\tfrac{\partial\theta}{\partial t}\in V_{0}$, and following standard computations, and invoking Liptschitz continuity of $f$ from Assumptions \ref{asp:01}--\ref{asp:02} and Lemma \ref{lem:local-f}, we find that
\begin{align*}
2\left\|\frac{\partial\theta}{\partial t}\right\|_{L^{2}}^{2}+\frac{d}{dt}\left\|\theta\right\|^{2}_{H^{1}}&=2F\left(p;\frac{\partial\theta}{\partial t}\right)-2F\left(p_{\rm{ms}};\frac{\partial\theta}{\partial t}\right)-2m\left(\frac{\partial\rho}{\partial t},\frac{\partial\theta}{\partial t}\right)\\
& =2F\left(p;\frac{\partial\theta}{\partial t}\right)-2F\left(\hat{p};\frac{\partial\theta}{\partial t}\right)+2F\left(\hat{p};\frac{\partial\theta}{\partial t}\right)-2F\left(p_{\rm{ms}};\frac{\partial\theta}{\partial t}\right)-2m\left(\frac{\partial\rho}{\partial t},\frac{\partial\theta}{\partial t}\right)\\
&\leq 2\left\|F\left(p;\frac{\partial\theta}{\partial t}\right)-2F\left(\hat{p};\frac{\partial\theta}{\partial t}\right)\right\|_{L^{2}}+2\left\|F\left(\hat{p};\frac{\partial\theta}{\partial t}\right)-2F\left(p_{\rm{ms}};\frac{\partial\theta}{\partial t}\right)\right\|_{L^{2}}\\
&\quad+2\left\|\frac{\partial\rho}{\partial t}\right\|_{L^{2}}\left\|\frac{\partial\theta}{\partial t}\right\|_{L^{2}}\\
&\leq \|\rho\|^{2}_{L^{2}}+\|\theta\|_{H^{1}}^{2}+\left\|\frac{\partial\rho}{\partial t}\right\|_{L^{2}}^{2}+ 2\left\|\frac{\partial\theta}{\partial t}\right\|_{L^{2}}^{2},
\end{align*}
we invoke Cauchy-Schwarz and Young's inequality to obtain the last inequality. Then, integrating for time, we get that
\[
\left\|\theta(t)\right\|^{2}_{H^{1}}\leq \|\theta(0)\|_{H^{1}}^{2}+\int_{0}^{t}\left(\|\rho\|^{2}_{L^{2}}+\left\|\frac{\partial\rho}{\partial t}\right\|_{L^{2}}^{2}\right)ds+\int_{0}^{t}\|\theta\|_{H^{1}}^{2}ds.
\]
Next, we use Gronwall's inequality and Lemma~\ref{lem:esti-proj} to arrive at
\begin{align*}
\left\|\theta(t)\right\|^{2}_{H^{1}}&\leq \|\theta(0)\|_{H^{1}}^{2}+C\int_{0}^{t}\left(\|\rho\|^{2}_{L^{2}}+\left\|\frac{\partial\rho}{\partial t}\right\|_{L^{2}}^{2}\right)ds\\
& \leq \|\theta(0)\|_{H^{1}}^{2}+C_{T}C_{\rm{ov}}^{2}\kappa_{\min}^{-2}\max_{1\leq i\leq N_v^H}\left(\lambda_{N_b^{S_{i}}+1}^{S_{i}}\right)^{-2},
\end{align*}
where $C_{T}$ is a constant depending on $T$. Finally, we use the triangle inequality to get, 
\[
\left\|p-p_{\rm{ms}}\right\|_{H^{1}}\leq \|(p-p_{\rm{ms}})(0)\|_{H^{1}}+C\left(C_{\rm{ov}}\kappa_{\min}^{-1}\max_{1\leq i\leq N_v^H}\left(\lambda_{N_b^{S_{i}}+1}^{S_{i}}\right)^{-1}+C_{\rm{ov}}^{1/2}\kappa_{\min}^{-1/2}\max_{1\leq i\leq N_v^H}\left(\lambda_{N_b^{S_{i}}+1}^{S_{i}}\right)^{-1/2}\right).
\]
This finishes the proof for \eqref{eq:thmH1main}. Estimate \eqref{eq:thmL2main} follows similar computations now using $v=\theta$ in \eqref{eq:th-h1} and the coercivity of bilinear form $a(\cdot,\cdot)$. This completes the proof.
\end{proof}

Observe that, from Theorem \ref{thm:estiH1}, we infer that
\begin{equation}
\label{eq:onepart}
\|(p-p_{\rm{ms}})(t^{n})\|_{H^{1}}\leq \|(p-p_{\rm{ms}})(0)\|_{H^{1}}+C\left(C_{\rm{ov}}\kappa_{\min}^{-1}\max_{1\leq i\leq N_v^H}\left(\lambda_{N_b^{S_{i}}+1}^{S_{i}}\right)^{-1}+C_{\rm{ov}}^{1/2}\kappa_{\min}^{-1/2}\max_{1\leq i\leq N_v^H}\left(\lambda_{N_b^{S_{i}}+1}^{S_{i}}\right)^{-1/2}\right),
\end{equation}
for $n=1,\dots,N_{\rm{t}}$.

Following \cite{henry1981geometric}, in virtue of $V$-ellipticity of $N$, we have $-N$ is a sectorial on $L^{2}(D)$ (uniformly in $h$), i.e., there exist constants $C>0$ and $\theta \in(\tfrac{1}{2}\pi,\pi)$ such that
\[
\|(\lambda I + N)^{-1}\|\leq \frac{C}{|\gamma|},\quad \gamma\in M_{\theta},
\]
where $M_{\theta}:\{\gamma\in\mathbb{C}:\lambda=\rho e^{i\phi},\rho>0,0\leq\phi\leq \theta\}$. The discrete operator is the infinitesimal of the exponential operator $e^{-tN}$ on $V$ such that
\[
e^{-tN}:=\frac{1}{2\pi}\int_{\Gamma}e^{t\gamma}(\gamma I+N)^{-1}d\gamma,\quad t>0,
\]
where $\Gamma$ denotes a path surrounding the spectrum of $-N$. The following result will be mainly used in this work.
\begin{prp}[Properties of the semigroup \cite{henry1981geometric}]
\label{prp:prop-semigroup}
Let $\alpha\geq 0$ and any given parameter $0\leq \beta\leq 1$, then there exist constants $C_{0},C_{1}>0$ such that
\begin{align*}
\|t^{\alpha}N^{\alpha}e^{-tN}\|_{L^{2}}\leq C_{0},\quad\mbox{for }t>0,\\
\left\|t^{\beta}N^{\beta}\sum_{j=1}^{n-1}e^{-jt N}\right\|_{0}\leq C_{1},\quad\mbox{for } t>0.
\end{align*}
\end{prp}

\begin{thm}
\label{thm:main-result}
Let $p\in V$ be the solution of \eqref{eq:fine_sys} and $p_{\rm{ms}}$ multiscale solution obtained from scheme \eqref{eq:mgsfem_ei}. Assume that the function $f$ fulfills Assumption~\ref{asp:02} and Lemma~\ref{lem:local-f}. Then, holds
\begin{equation}\label{eq:main-thm1}
\|\varepsilon^{n}\|_{H^{1}}\leq \|(p-p_{\rm{ms}})(0)\|_{H^{1}}+C\left\{\tau+C_{\rm{ov}}\kappa_{\min}^{-1}\max_{1\leq i\leq N_v^H}\left(\lambda_{N_b^{S_{i}}+1}^{S_{i}}\right)^{-1}+C_{\rm{ov}}^{1/2}\kappa_{\min}^{-1/2}\max_{1\leq i\leq N_v^H}\left(\lambda_{N_b^{S_{i}}+1}^{S_{i}}\right)^{-1/2}\right\} 
\end{equation}
where $\varepsilon^{n}:=p(t^{n})-p_{\rm{ms}}^{n}$, for $n=1,\dots,N_{\rm{t}}$, and $C>0$ is a constant that dependent on $u,\tfrac{\partial u}{\partial t},f$ and $T$.
\end{thm}

\begin{proof}
Using the triangle inequality, we have that,
\[
\|p(t^{n})-p_{\rm{ms}}^{n}\|_{H^{1}}\leq \|p(t^{n})-p_{\rm{ms}}(t^{n})\|_{H^{1}}+\|p_{\rm{ms}}(t^{n})-p_{\rm{ms}}^{n}\|_{H^{1}}=I_{1}+I_{2},
\]
where a bound for $I_{1}$ is estimated in Theorem~\ref{thm:estiH1}. Then, we focus on finding a bound for $I_{2}$. From expressions \eqref{eq:variation_of_constants} and \eqref{eq:mgsfem_ei}, follows that
\begin{equation}\label{eq:variation-ms-exp}
\begin{split}
p_{\rm{ms}}(t^{n}) &= e^{-\tau N_{0}}p_{\rm{ms}}(t^{n-1})+\int_{0}^{\tau}e^{-(\tau-s)N_{0}}F_{0}(p_{\rm{ms}}(t^{n-1}+s))ds\\
& = e^{-\tau N_{0}}p_{\rm{ms}}(t^{n-1})+\int_{0}^{\tau}e^{-(\tau-s)N_{0}}F_{0}(p(t^{n-1}+s))ds\\
 & \quad+\int_{0}^{\tau}e^{-(\tau-s)N_{0}}\left[F_{0}(p_{\rm{ms}}(t^{n-1}+s))-F_{0}(p(t^{n-1}+s))\right]ds,
\end{split}
\end{equation}
where $F_{0}$ is the $L^{2}$-orthogonal projection of $F$ in $V_{0}$. Next, using a Taylor expansion with integral remainder in the first integral to find 
\[
F_{0}(p(t^{n-1}+s)) = F_{0}(p(t^{n-1}))+\int_{0}^{s}F^{(1)}_{0}(p(t^{n-1}+\sigma))d\sigma.
\]
Then, substituting the expression above into \eqref{eq:variation-ms-exp} and using the recurrence relation \eqref{eq:varphi_definition} yield
\begin{equation}\label{eq:variation-ms-exp2}
\begin{split}
p_{\rm{ms}}(t^{n}) &= e^{-\tau N_{0}}p_{\rm{ms}}(t^{n-1})+\int_{0}^{\tau}e^{-(\tau-s)N_{0}}F_{0}(p(t^{n-1}))ds+\int_{0}^{\tau}e^{-(\tau-s)N_{0}}\int_{0}^{s}F^{(1)}_{0}(p(t^{n-1}+\sigma))d\sigma ds\\
 & \quad+\int_{0}^{\tau}e^{-(\tau-s)N_{0}}\left[F_{0}(p_{\rm{ms}}(t^{n-1}+s))-F_{0}(p(t^{n-1}+s))\right]ds,\\
& = e^{-\tau N_{0}}p_{\rm{ms}}(t^{n-1})+\tau\phi_{1}(-\tau N_{0})F_{0}(p(t^{n-1}))+\int_{0}^{\tau}e^{-(\tau-s)N_{0}}\int_{0}^{s}F^{(1)}_{0}(p(t^{n-1}+\sigma))d\sigma ds\\
 & \quad+\int_{0}^{\tau}e^{-(\tau-s)N_{0}}\left[F_{0}(p_{\rm{ms}}(t^{n-1}+s))-F_{0}(p(t^{n-1}+s))\right]ds.
\end{split}
\end{equation}
Now, in virtue of \eqref{eq:exp_euler_scheme} or \eqref{eq:mgmsfem_ei_2},  $p_{\rm{ms}}^{n}=e^{-\tau N_{0}}p_{\rm{ms}}^{n-1}+\tau\phi_{1}(-\tau N_{0})F_{0}(p_{\rm{ms}}^{n-1})$, we arrive at
\begin{align*}
p_{\rm{ms}}(t^{n})-p_{\rm{ms}}^{n}& = e^{-\tau N_{0}}[p_{\rm{ms}}(t^{n-1})-p_{\rm{ms}}^{n-1}]+\tau\phi_{1}(-\tau N_{0})\left[F_{0}(p(t^{n-1}))-F_{0}(p_{\rm{ms}}^{n-1})\right]\\
&\quad+\int_{0}^{\tau}e^{-(\tau-s)N_{0}}\int_{0}^{s}F^{(1)}_{0}(p(t^{n-1}+\sigma))d\sigma ds\\
 & \quad+\int_{0}^{\tau}e^{-(\tau-s)N_{0}}\left[F_{0}(p_{\rm{ms}}(t^{n-1}+s))-F_{0}(p(t^{n-1}+s))\right]ds\\
&=e^{-\tau N_{0}}[p_{\rm{ms}}(t^{n-1})-p_{\rm{ms}}^{n-1}]+\tau\phi_{1}(-\tau N_{0})\left[F_{0}(p(t^{n-1}))-F_{0}(p_{\rm{ms}}^{n-1})\right]+\epsilon^{n}.
\end{align*}
We shall redefine the expression above by
\begin{equation}\label{eq:varepsilon-main}
\varepsilon^{n}_{\rm{ms}}:=p_{\rm{ms}}(t^{n})-p_{\rm{ms}}^{n}=e^{-\tau N_{0}}\varepsilon^{n-1}_{\rm{ms}}+\tau\phi_{1}(-\tau N_{0})\left[F_{0}(p(t^{n-1}))-F_{0}(p_{\rm{ms}}^{n-1})\right]+\epsilon^{n}.
\end{equation}
Following \cite{hochbruck2005explicit,poveda2024second}, we obtain by recursion substitution that
\[
\varepsilon^{n}_{\rm{ms}}=\tau\sum_{j=1}^{n-1}e^{-(n-j-1)\tau N_{0}}\phi_{1}(-\tau N_{0})\left[F_{0}(p(t^{j}))-F_{0}(p^{j})\right]+\sum_{j=0}^{n-1}e^{-j\tau N_{0}}\epsilon^{n-j}.
\] 
Thus, we get for $I_{2}$
\begin{equation}\label{eq:I2}
I_{2}\leq\left\|\tau\sum_{j=1}^{n-1}e^{-(n-j-1)\tau N_{0}}\phi_{1}(-\tau N_{0})\left[F_{0}(p(t^{j}))-F_{0}(p^{j})\right]\right\|_{H^{1}}+\left\|\sum_{j=0}^{n-1}e^{-j\tau N_{0}}\epsilon^{n-j}\right\|_{H^{1}}= I_{2,1}+I_{2,2}.
\end{equation}
For $I_{2,1}$, we have
\begin{align*}
I_{2,1}&\leq \left\|\tau\phi_{1}(\tau N_{0})[F_{0}(p(t^{n-1}))-F_{0}(p^{n-1})]\right\|_{H^{1}}+\left\|\tau\sum_{j=0}^{n-2}e^{-(n-j-1)\tau N_{0}}\phi_{1}(-\tau N_{0})\left[F_{0}(p(t^{j}))-F_{0}(p^{j})\right]\right\|_{H^{1}}\\
&= I_{2,1,1}+I_{2,1,2}.
\end{align*}
Note that $N_{0}$ is an operator associated with the bilinear form $a(\cdot,\cdot)$, then $N_{0}=N_{0}^{1/2}N_{0}^{1/2}$. In addition, we get a relation between $H^{1}$-, and $L^{2}$-norm over multiscale space $V_{0}$, in which there exists a constant $\alpha>0$, such that $\tfrac{1}{\alpha}\|N^{1/2}_{0}v\|_{L^{2}}\leq\|v\|_{H^{1}}\leq \alpha\|N^{1/2}_{0}v\|_{L^{2}}$. Then, for $I_{2,1,1}$, from Proposition~\ref{prp:prop-semigroup} and Lemma~\ref{lem:local-f}, we obtain
\begin{equation}\label{eq:I211}
\begin{split}
I_{2,1,1} &\leq \alpha \left\|\tau N_{0}^{1/2}\phi_{1}(\tau N_{0})\right\|_{L^{2}}\left\|F_{0}(p(t^{n-1}))-F_{0}(p^{n-1})\right\|_{L^{2}}\\
&=\alpha\left\|N^{1/2}_{0}\int_{0}^{\tau}e^{-(\tau-s)N_{0}}ds\right\|_{L^{2}}\left\|F_{0}(p(t^{n-1}))-F_{0}(p^{n-1})\right\|_{L^{2}}\\
&\leq C\tau\sup_{0\leq s\leq \tau}\left\|N_{0}^{1/2}e^{-(\tau-s)N_{0}}ds\right\|_{L^{2}}\|p(t^{n-1})-p^{n-1}\|_{H^{1}}\\
&\leq C\tau^{1/2}\|p(t^{n-1})-p^{n-1}\|_{H^{1}}=C\tau^{1/2}\|\varepsilon^{n-1}_{\rm{ms}}\|_{H^{1}}.
\end{split}
\end{equation}
For $I_{2,1,2}$, using the fact $\|\phi_{1}(-\tau N_{0})\|_{L^{2}}\leq C$, and the relation between norms, we get
\begin{align*}
I_{2,1,2}&\leq \left\|\tau\sum_{j=0}^{n-2}e^{-(n-j-1)\tau N_{0}}\phi_{1}(-\tau N_{0})\left[F_{0}(p(t^{j}))-F_{0}(p_{\rm{ms}}(t^{j}))\right]\right\|_{H^{1}}\\
&\quad+\left\|\tau\sum_{j=0}^{n-2}e^{-(n-j-1)\tau N_{0}}\phi_{1}(-\tau N_{0})\left[F_{0}(p_{\rm{ms}}(t^{j}))-F_{0}(p^{j})\right]\right\|_{H^{1}}\\
& \leq \left\|\tau N_{0}^{1/2}\sum_{j=0}^{n-2}e^{-(n-j-1)\tau N_{0}}\left[F_{0}(p(t^{j}))-F_{0}(p_{\rm{ms}}(t^{j}))\right]\right\|_{L^{2}}\\
&\quad+\tau\sum_{j=0}^{n-2}\left\|N_{0}^{1/2}e^{-(n-j-1)\tau N_{0}}\left[F_{0}(p_{\rm{ms}}(t^{j}))-F_{0}(p^{j})\right]\right\|_{L^{2}}.
\end{align*}
Invoking Proposition~\ref{prp:prop-semigroup} and Lemma~\ref{lem:local-f} we find
\begin{equation}\label{eq:I212}
\begin{split}
I_{2,1,2}& \leq C\sup_{0\leq t\leq T}\left\|F_{0}(p(t))-F_{0}(p_{\rm{ms}}(t))\right\|_{L^{2}}+\tau\sum_{j=0}^{n-2}\left(t^{n-j-1}\right)^{-1/2}\left\|F_{0}(p_{\rm{ms}}(t^{j}))-F_{0}(p^{j})\right\|_{L^{2}}\\
& \leq C\sup_{0\leq t \leq T}\|p(t)-p_{\rm{ms}}(t)\|_{H^{1}}+\tau\sum_{j=0}^{n-2}\left(t^{n-j-1}\right)^{-1/2}\left\|\varepsilon^{j}\right\|_{H^{1}}.
\end{split}
\end{equation}
The first term in the expression above is bounded by Theorem \ref{thm:estiH1}. Thus, gathering \eqref{eq:I211} and \eqref{eq:I212}, we get 
\begin{equation}\label{eq:I21}
I_{2,1}\leq C\tau^{1/2}\|\varepsilon^{n-1}_{\rm{ms}}\|_{H^{1}}+\tau\sum_{j=0}^{n-2}\left(t^{n-j-1}\right)^{-1/2}\left\|\varepsilon^{j}_{\rm{ms}}\right\|_{H^{1}}+ C(\text{by Theorem~\ref{thm:estiH1}}).
\end{equation}
Following a similar framework given in \cite{hochbruck2005explicit,poveda2024second}, using Theorem~\ref{thm:estiH1} and Proposition~\ref{prp:prop-semigroup}, we arrive at
\begin{equation}\label{eq:I22}
I_{2,2}\leq \tau \sup_{0\leq t\leq T}\|F_{0}^{(1)}(p(t))\|_{H^{1}}+C(\text{by Theorem~\ref{thm:estiH1}}).
\end{equation}
Then, combining \eqref{eq:I21} and \eqref{eq:I22} into \eqref{eq:I2}, we arrive at
\begin{equation}\label{eq:I2main}
\begin{split}
I_{2}&\leq C\left\{\tau^{1/2}\|\varepsilon^{n-1}_{\rm{ms}}\|_{H^{1}}+\tau\sum_{j=0}^{n-2}\left(t^{n-j-1}\right)^{-1/2}\left\|\varepsilon^{j}_{\rm{ms}}\right\|_{H^{1}}+\tau\right\}+\|(p-p_{\rm{ms}})(0)\|_{H^{1}}\\
& \quad +C\left(C_{\rm{ov}}\kappa_{\min}^{-1}\max_{1\leq i\leq N_v^H}\left(\lambda_{N_b^{S_{i}}+1}^{S_{i}}\right)^{-1}+C_{\rm{ov}}^{1/2}\kappa_{\min}^{-1/2}\max_{1\leq i\leq N_v^H}\left(\lambda_{N_b^{S_{i}}+1}^{S_{i}}\right)^{-1/2}\right)\\
&=C\left\{\tau\sum_{j=0}^{n-1}\left(t^{n-j}\right)^{-1/2}\left\|\varepsilon^{j}_{\rm{ms}}\right\|_{H^{1}}+\tau\right\}+\|(p-p_{\rm{ms}})(0)\|_{H^{1}}\\
&\quad+C\left(C_{\rm{ov}}\kappa_{\min}^{-1}\max_{1\leq i\leq N_v^H}\left(\lambda_{N_b^{S_{i}}+1}^{S_{i}}\right)^{-1}+C_{\rm{ov}}^{1/2}\kappa_{\min}^{-1/2}\max_{1\leq i\leq N_v^H}\left(\lambda_{N_b^{S_{i}}+1}^{S_{i}}\right)^{-1/2}\right).
\end{split}
\end{equation}
Finally, we use a discrete version of Gronwall's inequality (see, for instance, \cite{dixon1986weak}) to find 
\[
I_{2}\leq \|(p-p_{\rm{ms}})(0)\|_{H^{1}}+C\left\{\tau+C_{\rm{ov}}\kappa_{\min}^{-1}\max_{1\leq i\leq N_v^H}\left(\lambda_{N_b^{S_{i}}+1}^{S_{i}}\right)^{-1}+C_{\rm{ov}}^{1/2}\kappa_{\min}^{-1/2}\max_{1\leq i\leq N_v^H}\left(\lambda_{N_b^{S_{i}}+1}^{S_{i}}\right)^{-1/2}\right\}.
\]
Finally, we prove the desired assertion using Theorem~\ref{thm:estiH1} for $I_{1}$ and the estimate above. 
\end{proof}
We can obtain the following error estimates in $L^{2}$-norm using similar manipulations.

\begin{thm}
\label{thm:estiL2}
Let $p\in V$ be the solution of \eqref{eq:fine_sys} and $p_{\rm{ms}}$ multiscale solution obtained from scheme \eqref{eq:mgsfem_ei}. Assume that the function $f$ fulfills Assumption~\ref{asp:02} and Lemma~\ref{lem:local-f}. Then, holds
\begin{equation}\label{eq:main-thm1}
\|\varepsilon^{n}\|_{L^{2}}\leq \|(p-p_{\rm{ms}})(0)\|_{L^{2}}+C\left\{\tau+C_{\rm{ov}}\kappa_{\min}^{-1}\max_{1\leq i\leq N_v^H}\left(\lambda_{N_b^{S_{i}}+1}^{S_{i}}\right)^{-1}\right\},
\end{equation}
where $\varepsilon^{n}:=p(t^{n})-p_{\rm{ms}}^{n}$, for $n=1,\dots,N_{\rm{t}}$, and $C>0$ is a constant that dependent on $u,\tfrac{\partial u}{\partial t},f$ and $T$.
\end{thm}

\section{Numerical results}\label{sec:numerical_results}
This section presents numerical results to check the proposed meshfree multiscale approach with exponential integration. For this purpose, we consider two model problems based on \eqref{eq:fine_sys} in the computational domain $\Omega = [0, 1]\times [0, 1]$. \corr{Note that the proposed multiscale approach can be expanded to solve three-dimensional problems. MFGMsFEM has already been successfully used for complex three-dimensional problems \cite{nikiforov2023meshfree, nikiforov2023modeling}. The addition of exponential integration can further enhance its efficiency due to a more stable temporal approximation.}

The first problem represents the linear case, and the second is semilinear with a nonlinear right-hand side. For all problems, we set the same boundary and initial conditions. We set a Dirichlet boundary condition on the domain boundary with $p_D = 0$. For an initial condition, we use the heterogeneous function $p_0 = x(1-x)y(1-y)$. The duration of the modeled process is $T = 0.2$.

To numerically analyze the performance of the proposed multiscale approach, we solve the model problems using different methods. As a reference method, we consider a finite element method with linear basis functions on a uniform fine grid with 10201 vertices and 20000 triangular cells. We combine it with the backward Euler method with 30000 time steps. As a standard multiscale approach, we consider the meshfree Generalized Multiscale Finite Element Method with 121 coarse-scale points combined with the backward Euler method, which we will refer to as MFGMsFEM-FD. Finally, we solve the problems using the proposed meshfree Generalized Multiscale Finite Element Method with exponential time integration, which we will denote as MFGMsFEM-EI. The time steps count for the multiscale approaches will be varied to investigate the stability of the solutions.


To estimate the accuracy of the multiscale approaches, we use the following relative $L^2$ and $H^1$ errors
\[ 
{||e||}_{L^2} = \frac{{||p_1-p_2||}_{L^2}^w}{{||p_1||}_{L^2}^w}\cdot100\%, \quad
{||e||}_{H^1} = \frac{{||p_1-p_2||}_{H^1}^w}{{||p_1||}_{H^1}^w}\cdot100\%.
\]
Here, ${||p||}_{L^2}^w = \sqrt{\int_{\Omega} \kappa(x) p^2 \,dx}$, and ${||p||}_{H^1}^w = \sqrt{\int_{\Omega} \kappa(x) \grad p \cdot \grad p \,dx}$, where $p_1$ denotes the reference finite-element solution, and $p_2$ is a multiscale solution (MFGMsFEM-FD or MFGMsFEM-EI).

The numerical implementation of the reference and multiscale methods is based on the open-source computational package FEniCS \cite{logg2012automated}. The FEniCS package provides a flexible and powerful toolkit for implementing the finite element method, working with matrices and vectors, and solving sparse linear systems. The obtained numerical results are visualized using the open-source ParaView visualization software \cite{ahrens2005paraview}.

\subsection{Linear case}\label{subsec:linear_results}
This subsection considers the linear case, where the right-hand side is $f(p) = 0$. Fig. \ref{ris:HCC1} depicts the heterogeneous permeability coefficient $\kappa(x)$, where in the green region $\kappa(x) = \kappa_1 = 1$ and in the yellow region $\kappa(x) = \kappa_2 = 10^4$. Thus, we consider a multiscale medium with a contrast of $10^4$. 

\begin{figure}[!hbt]
\begin{center}
\includegraphics[width=0.4\linewidth]{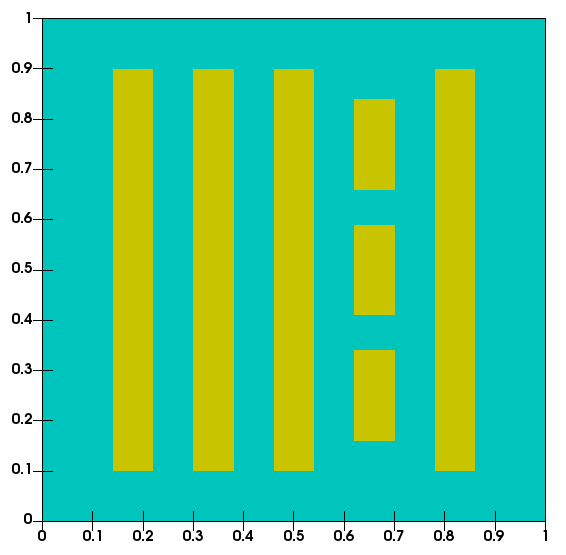}
\end{center}
\caption{High-contrast permeability coefficient $\kappa(x)$ for the linear case, where $\kappa_1 = 1$ is green and $\kappa_2 = 10^4$ is yellow.}
\label{ris:HCC1}
\end{figure}

\begin{figure}[!hbt]
\begin{center}
\begin{minipage}[h]{0.325\linewidth}
\center{\includegraphics[width=\linewidth]{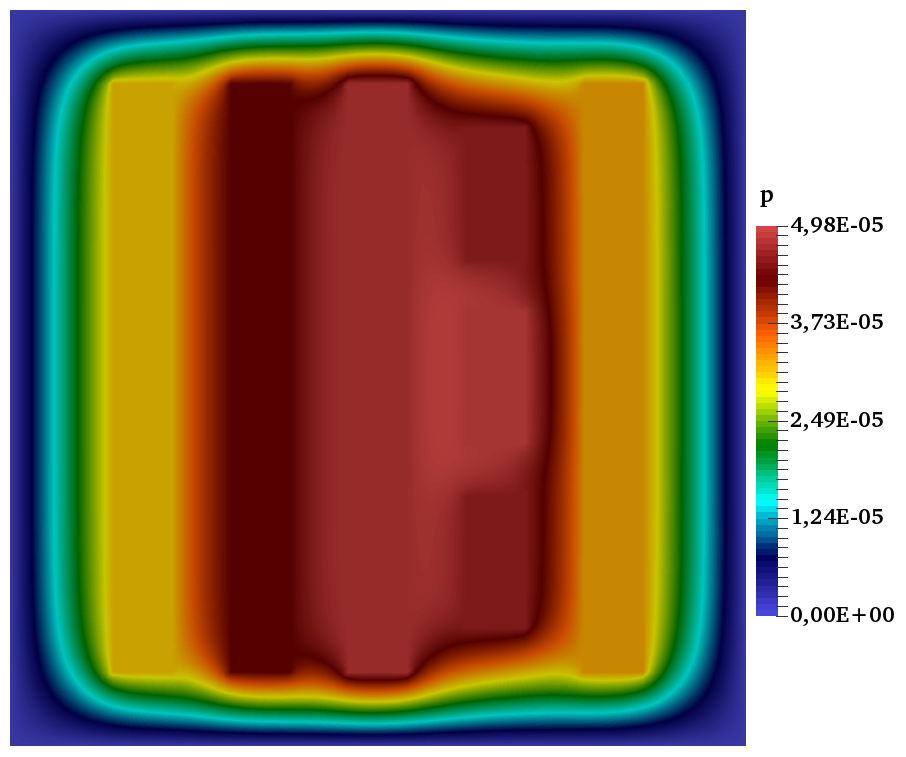}}
\end{minipage}
\begin{minipage}[h]{0.325\linewidth}
\center{\includegraphics[width=\linewidth]{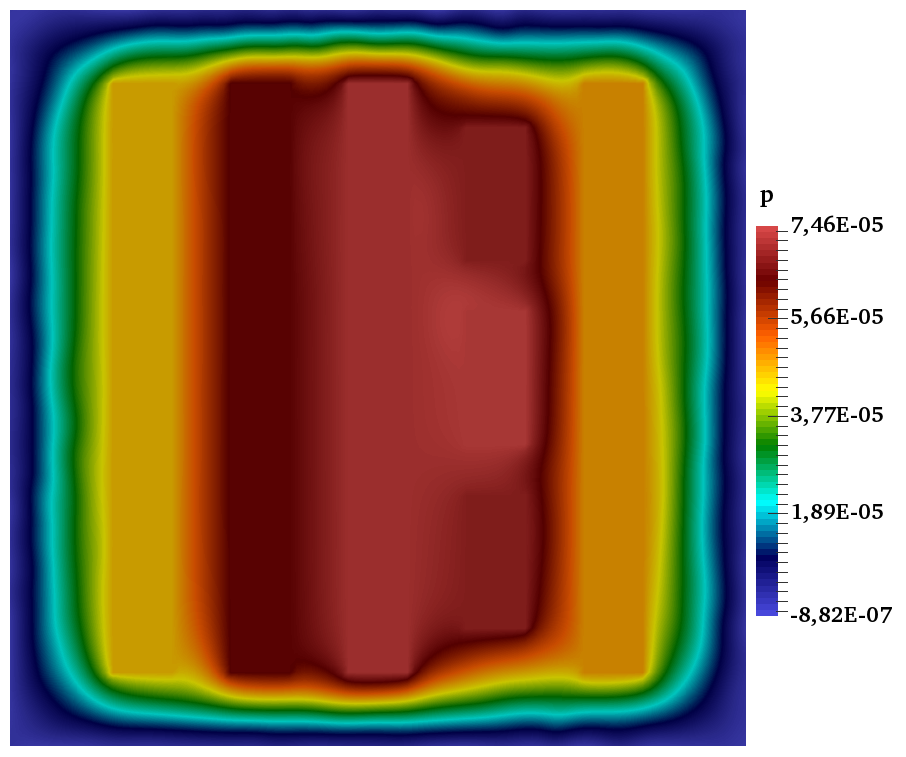}}
\end{minipage}
\begin{minipage}[h]{0.325\linewidth}
\center{\includegraphics[width=\linewidth]{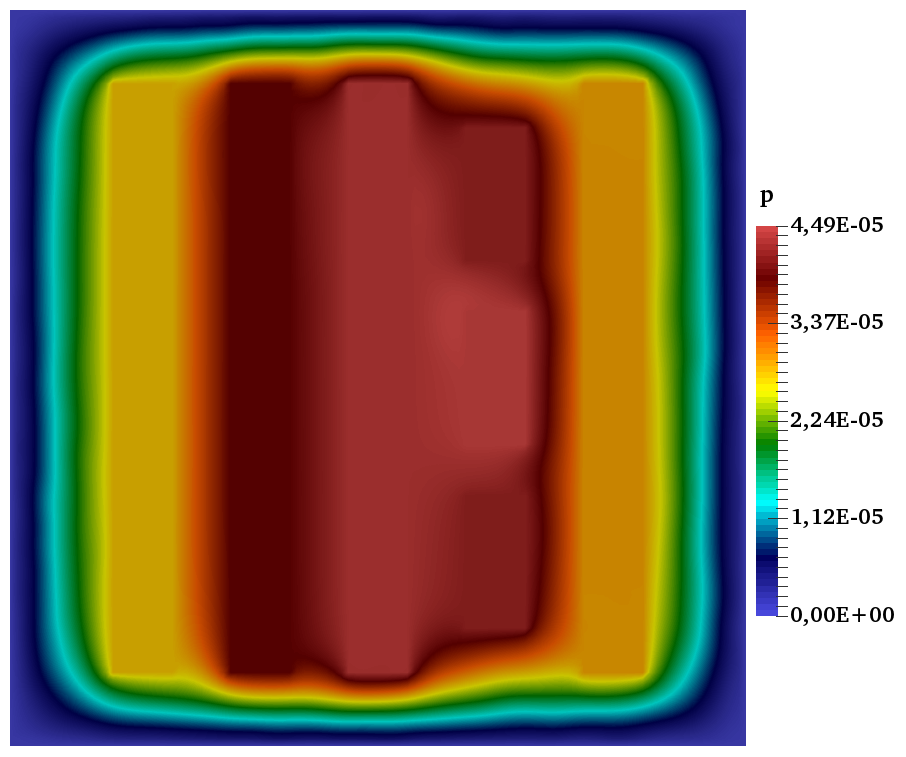}}
\end{minipage}
\end{center}
\caption{Pressure distributions at the final time for the linear case. Left: reference solution. Center: MFGMsFEM-FD solution. Right: MFGMsFEM-EI solution. The multiscale approaches' parameters are $M = 5$, $N_t = 50$, and $\gamma = 3$.}
\label{ris:sol1}
\end{figure}

In Fig. \ref{ris:sol1}, we present the pressure distribution at the final time. We depict the reference, MFGMsFEM-FD, and MFGMsFEM-EI solutions from left to right. For multiscale methods, we use $M=5$ basis functions, $N_t = 50$ time steps, and coverage parameter $\gamma = 3$. We can see that both multiscale methods provide a general dynamics of the reference solution. However, the MFGMsFEM-FD solution contains slight oscillatory regions near boundaries. However, the main issue of the MFGMsFEM-FD is that the solution's values significantly differ from the reference one. In contrast, the values of the MFGMsFEM-EI solution are almost the same as the reference solution. Thus, MFGMsFEM-EI accurately approximates the reference solution with fewer degrees of freedom and time steps.



\begin{figure}[!hbt]
\begin{center}
\begin{minipage}[h]{0.44\linewidth}
\center{\includegraphics[width=\linewidth]{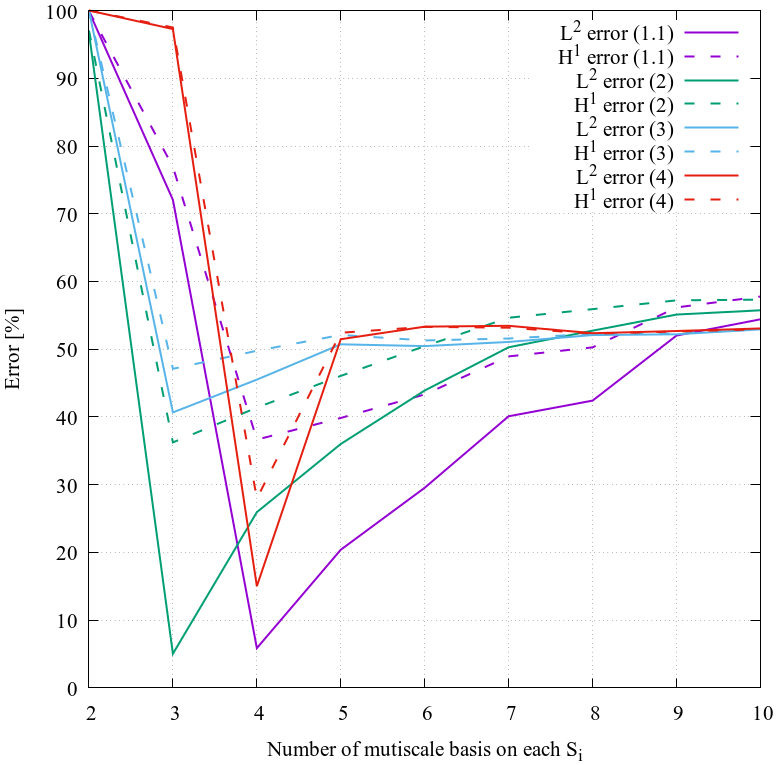}}
\end{minipage}
\begin{minipage}[h]{0.44\linewidth}
\center{\includegraphics[width=\linewidth]{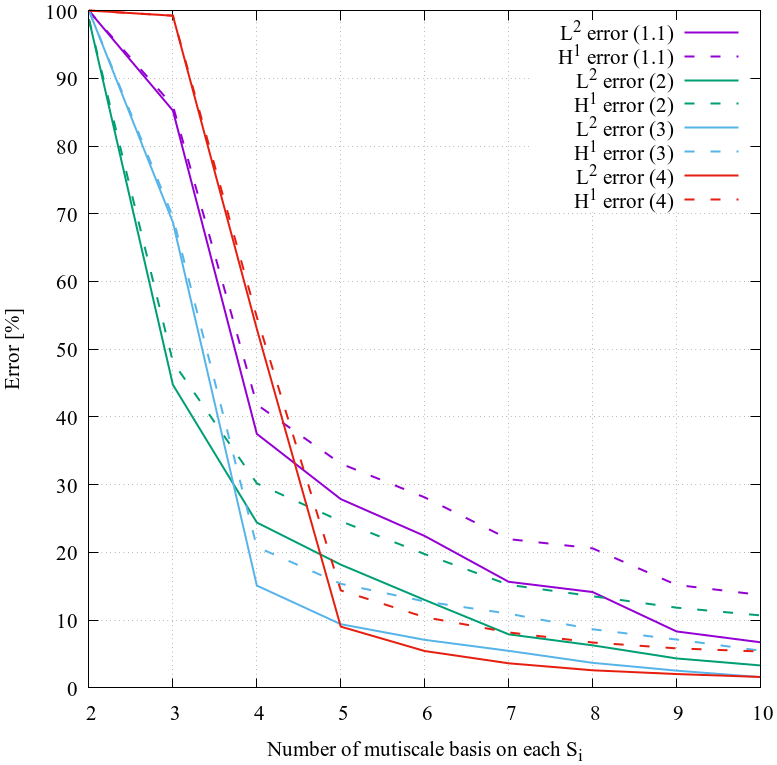}}
\end{minipage}
\end{center}
\caption{The relative $L^2$ and $H^1$ errors between the reference solution and the multiscale solutions using $N_t = 50$ and various $M$ and $\gamma$ at the final time $T = 0.2$ for the linear case. Left: MFGMsFEM-FD. Right: MFGMsFEM-EI.}
\label{ris:error11}
\end{figure}

Let us proceed with the error analysis of multiscale approaches. To investigate convergence on basis functions $M$ and coverage parameter $\gamma$, we fix the number of time steps $N_t = 50$ for both multiscale methods. We present the obtained errors at the final time as plots in Fig. \ref{ris:error11}. From left to right, we depict the plots for MFGMsFEM-FD and MFGMsFEM-EI. Let us consider the results for MFGMsFEM-FD. We see that the general dynamics of the plots for the $L^2$ and $H^1$ errors are much the same. One can see that increasing the number of basis functions does not reduce the error. The value of gamma affects the initial decay of the error, which then increases. The lowest values of $L^2$ errors are reached at $\gamma=2$ and $M=3$, while the lowest $H^1$ error is achieved at $\gamma=4$ and $M=4$. Table \ref{tab:tab11} presents the exact error values for different combinations of $\gamma$ and $M$.

Next, let us analyze the errors of MFGMsFEM-EI (right-hand side of Fig. \ref{ris:error11}). We can observe the convergence of the multiscale method on basis functions. The general dynamics of the $L^2$ and $H^1$ error plots largely coincide. One can observe that the parameter $\gamma$ affects the error decay rate. Except for the case $\gamma=1.1$, a larger value of $\gamma$ leads to a slower start of the error decay. However, after some time, the bigger $\gamma$ errors can surpass those of smaller $\gamma$, whose decay started earlier. In general, we can say that the most optimal values seem to be $\gamma=2$ and $\gamma=3$, which are often used in meshfree methods. Table \ref{tab:tab12} presents the exact error values of MFGMsFEM-EI. Therefore, the proposed multiscale approach with exponential integration provides high accuracy and convergence on basis functions for various $\gamma$ cases.


\begin{figure}[!hbt]
\begin{center}
\begin{minipage}[h]{0.44\linewidth}
\center{\includegraphics[width=\linewidth]{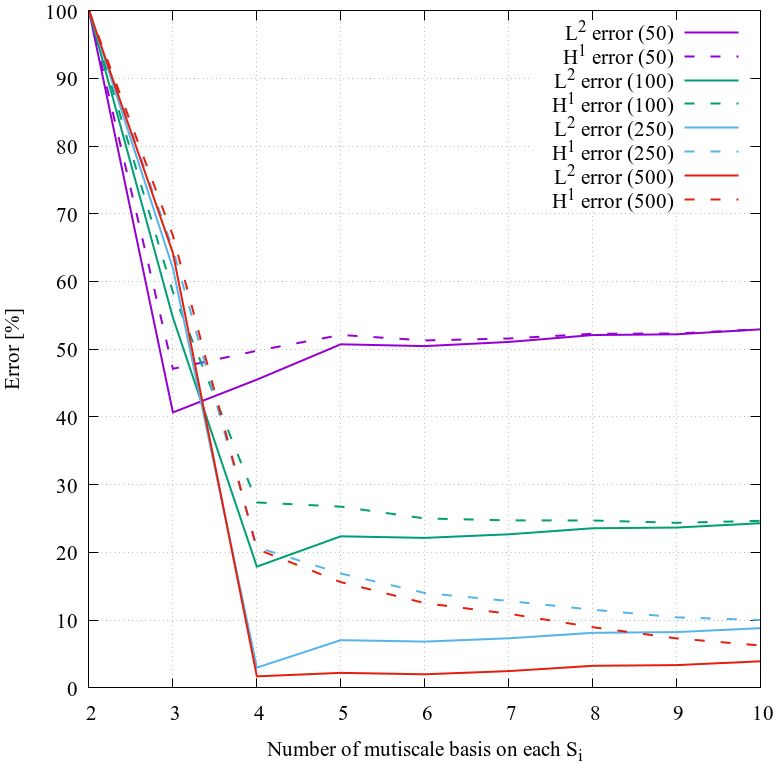}}
\end{minipage}
\begin{minipage}[h]{0.44\linewidth}
\center{\includegraphics[width=\linewidth]{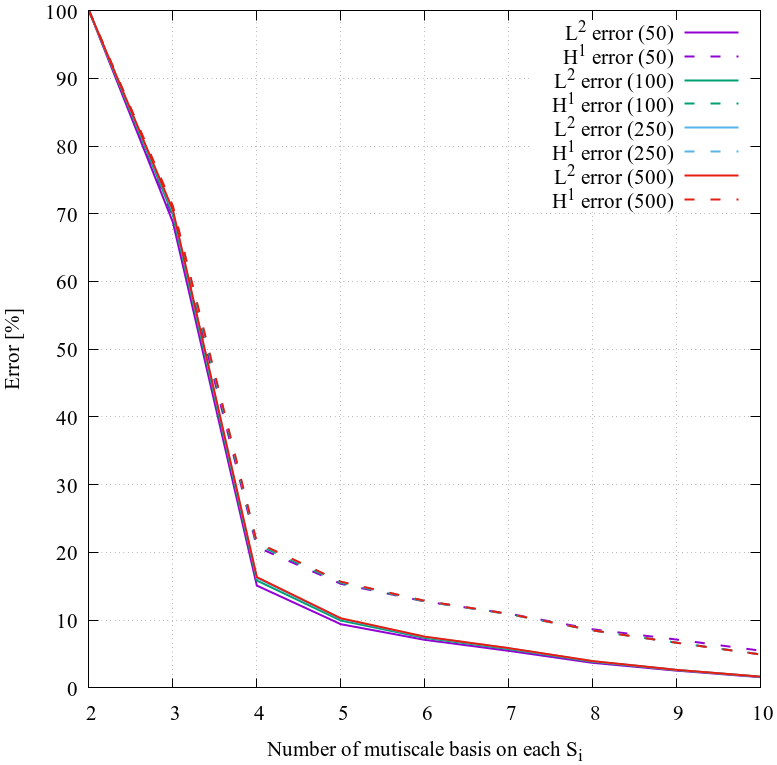}}
\end{minipage}
\end{center}
\caption{The relative $L^2$ and $H^1$ errors between the reference solution and the multiscale solutions using $\gamma = 3$ and various $M$ and $N_t$ at the final time $T = 0.2$ for the linear case. Left: MFGMsFEM-FD. Right: MFGMsFEM-EI.}
\label{ris:error12}
\end{figure}

To analyze the errors depending on the number of time steps, we fix the parameter $\gamma=3$. Fig. \ref{ris:error12} presents plots of relative errors at different numbers of time steps for MFGMsFEM-FD (left) and MFGMsFEM-EI (right). Let us consider the relative errors for MFGMsFEM-FD. One can observe that an increase in the number of time steps leads to a decrease in the errors. However, after the initial decay around 3-5 basis functions, the errors remain at the same level. The exception is the $H^1$ errors for $N_t \geqslant 250$, which decrease as the number of basis functions increases. Thus, the accuracy and convergence of MFGMsFEM-FD depend strongly on the time step size. Table \ref{tab:tab13} shows the exact values of the relative $L^2$ and $H^1$ errors.

Next, let us consider the errors of MFGMsFEM-EI depending on the number of time steps. In all the considered cases, we observe an apparent convergence on the basis functions. Moreover, the number of temporal steps has almost no effect on the accuracy of the multiscale method. The general dynamics of $L^2$ and $H^1$ errors largely coincide. Table \ref{tab:tab14} shows the exact values of the errors for the considered numbers of time steps. Thus, MFGMsFEM-EI achieves high accuracy even when large time steps are used, ensuring convergence on basis functions.

\begin{table}[!hbt]
\begin{center}
\begin{tabular}{c| c| c| c| c| c| c| c| c|}
\multirow{2}{*}{$M$} & \multicolumn{2}{c|}{$\gamma = 1.1$} & \multicolumn{2}{c|}{$\gamma = 2$} & \multicolumn{2}{c|}{$\gamma = 3$} & \multicolumn{2}{c|}{$\gamma = 4$} \\
 & ${||e||}_{L^2}$ & ${||e||}_{H^1}$ & ${||e||}_{L^2}$ & ${||e||}_{H^1}$ & ${||e||}_{L^2}$ & ${||e||}_{H^1}$ & ${||e||}_{L^2}$ & ${||e||}_{H^1}$ \\
\hline
2 & 99.883 & 99.907 & 96.732 & 97.054 & 100 & 100 & 100 & 100 \\
3 & 72.053 & 76.940 & 5.045  & 36.246 & 40.681 & 47.111 & 97.27  & 97.500 \\
4 & 5.876  & 36.664 & 25.953 & 41.387 & 45.522 & 49.780 & 15.017 & 27.959 \\
5 & 20.399 & 39.860 & 36.009 & 46.062 & 50.731 & 52.108 & 51.484 & 52.438 \\
6 & 29.554 & 43.325 & 43.908 & 50.456 & 50.463 & 51.297 & 53.338 & 53.31  \\
7 & 40.112 & 48.947 & 50.292 & 54.649 & 51.076 & 51.588 & 53.465 & 53.183 \\
8 & 42.418 & 50.282 & 52.742 & 55.923 & 52.094 & 52.280 & 52.372 & 52.306 \\
9 & 52.023 & 56.171 & 55.123 & 57.221 & 52.209 & 52.327 & 52.695 & 52.613 \\
10 & 54.419 & 57.777 & 55.75 & 57.323 & 52.939 & 52.953 & 53.072 & 52.955
\end{tabular}
\end{center}
\caption{The relative $L^2$ and $H^1$ errors between the reference solution and the MFGMsFEM-FD solution using $N_t = 50$ and various $M$ and $\gamma$ at the final time $T = 0.2$ for the linear case.}
\label{tab:tab11}
\end{table}

\begin{table}[!hbt]
\begin{center}
\begin{tabular}{c| c| c| c| c| c| c| c| c|}
\multirow{2}{*}{$M$} & \multicolumn{2}{c|}{$\gamma = 1.1$} & \multicolumn{2}{c|}{$\gamma = 2$} & \multicolumn{2}{c|}{$\gamma = 3$} & \multicolumn{2}{c|}{$\gamma = 4$} \\
 & ${||e||}_{L^2}$ & ${||e||}_{H^1}$ & ${||e||}_{L^2}$ & ${||e||}_{H^1}$ & ${||e||}_{L^2}$ & ${||e||}_{H^1}$ & ${||e||}_{L^2}$ & ${||e||}_{H^1}$ \\
\hline
2 & 99.958 & 99.958 & 98.712 & 98.678 & 100    & 100    & 100    & 100\\
3 & 85.239 & 85.955 & 44.759 & 48.14  & 68.75  & 69.514 & 99.237 & 99.239\\
4 & 37.500 & 41.826 & 24.42  & 30.198 & 15.108 & 20.808 & 52.986 & 54.729\\
5 & 27.880 & 33.118 & 18.198 & 24.618 & 9.411  & 15.374 & 9.051  & 14.403\\
6 & 22.428 & 28.131 & 12.988 & 19.758 & 7.098  & 12.77  & 5.444  & 10.452\\
7 & 15.668 & 21.975 & 7.916  & 15.229 & 5.478  & 10.961 & 3.638  & 8.198\\
8 & 14.161 & 20.608 & 6.292  & 13.537 & 3.702  & 8.668  & 2.605  & 6.711\\
9 & 8.342  & 15.215 & 4.356  & 11.856 & 2.555  & 7.149  & 2.055  & 5.843\\
10 & 6.735 & 13.728 & 3.318  & 10.712 & 1.586  & 5.495  & 1.627  & 5.384
\end{tabular}
\end{center}
\caption{The relative $L^2$ and $H^1$ errors between the reference solution and the MFGMsFEM-EI solution using $N_t = 50$ and various $M$ and $\gamma$ at the final time $T = 0.2$ for the linear case.}
\label{tab:tab12}
\end{table}

\begin{table}[!hbt]
\begin{center}
\begin{tabular}{c| c| c| c| c| c| c| c| c|}
\multirow{2}{*}{$M$} & \multicolumn{2}{c|}{$N_t = 50$} & \multicolumn{2}{c|}{$N_t = 100$} & \multicolumn{2}{c|}{$N_t = 250$} & \multicolumn{2}{c|}{$N_t = 500$} \\
 & ${||e||}_{L^2}$ & ${||e||}_{H^1}$ & ${||e||}_{L^2}$ & ${||e||}_{H^1}$ & ${||e||}_{L^2}$ & ${||e||}_{H^1}$ & ${||e||}_{L^2}$ & ${||e||}_{H^1}$ \\
\hline
2 & 100    & 100    & 100    & 100     & 100    & 100    & 100    & 100    \\
3 & 40.681 & 47.111 & 54.682 & 58.461  & 61.963 & 64.783 & 64.209 & 66.773 \\
4 & 45.522 & 49.780 & 17.913 & 27.386  & 3.006  & 20.766 & 1.714  & 20.496 \\
5 & 50.731 & 52.108 & 22.389 & 26.767  & 7.059  & 16.895 & 2.226  & 15.626 \\
6 & 50.463 & 51.297 & 22.152 & 25.023  & 6.84   & 14.027 & 2.012  & 12.519 \\
7 & 51.076 & 51.588 & 22.685 & 24.739  & 7.327  & 12.835 & 2.484  & 10.954 \\
8 & 52.094 & 52.280 & 23.575 & 24.72   & 8.141  & 11.55  & 3.273  & 8.969  \\
9 & 52.209 & 52.327 & 23.681 & 24.381  & 8.242  & 10.419 & 3.371  & 7.317  \\
10 & 52.939 & 52.953 & 24.319 & 24.669 & 8.826  & 10.022 & 3.938  & 6.269  
\end{tabular}
\end{center}
\caption{The relative $L^2$ and $H^1$ errors between the reference solution and the MFGMsFEM-FD solution using $\gamma = 3$ and various $M$ and $N_t$ at the final time $T = 0.2$ for the linear case.}
\label{tab:tab13}
\end{table}

\begin{table}[!hbt]
\begin{center}
\begin{tabular}{c| c| c| c| c| c| c| c| c|}
\multirow{2}{*}{$M$} & \multicolumn{2}{c|}{$N_t = 50$} & \multicolumn{2}{c|}{$N_t = 100$} & \multicolumn{2}{c|}{$N_t = 250$} & \multicolumn{2}{c|}{$N_t = 500$} \\
 & ${||e||}_{L^2}$ & ${||e||}_{H^1}$ & ${||e||}_{L^2}$ & ${||e||}_{H^1}$ & ${||e||}_{L^2}$ & ${||e||}_{H^1}$ & ${||e||}_{L^2}$ & ${||e||}_{H^1}$ \\
\hline
2 & 100    & 100    & 100    & 100    & 100    & 100    & 100    & 100    \\
3 & 68.75  & 69.514 & 69.824 & 70.481 & 70.374 & 70.983 & 70.542 & 71.137 \\
4 & 15.108 & 20.808 & 15.857 & 21.184 & 16.238 & 21.417 & 16.354 & 21.492 \\
5 & 9.411  & 15.374 & 9.947  & 15.523 & 10.212 & 15.648 & 10.292 & 15.691 \\
6 & 7.098  & 12.74  & 7.39   & 12.794 & 7.531  & 12.848 & 7.573  & 12.868 \\
7 & 5.478  & 10.961 & 5.733  & 10.892 & 5.851  & 10.924 & 5.886  & 10.938 \\
8 & 3.702  & 8.668  & 3.867  & 8.497  & 3.94   & 8.497  & 3.961  & 8.502  \\
9 & 2.555  & 7.148  & 2.633  & 6.684  & 2.666  & 6.673  & 2.675  & 6.674  \\
10 & 1.586 & 5.495  & 1.64   & 4.962  & 1.662  & 4.942  & 1.668  & 4.941  
\end{tabular}
\end{center}
\caption{The relative $L^2$ and $H^1$ errors between the reference solution and the MFGMsFEM-EI solution using $\gamma = 3$ and various $M$ and $N_t$ at the final time $T = 0.2$ for the linear case.}
\label{tab:tab14}
\end{table}

Therefore, we can conclude that MFGMsFEM-EI in the considered linear case allows us to achieve high accuracy and convergence on the basis functions for different $\gamma$ and time step sizes. While MFGMsFEM-FD requires a small time step size and does not always provide error reduction as the number of basis functions increases.

\FloatBarrier
\subsection{Semilinear case}\label{subsec:semilinear_results}
Let us now consider the semilinear case in which the right-hand side is a nonlinear function given as $f(p) = -p(1-p)(1+p)$. The nonlinear term adds further computational difficulties, especially for time discretization, so this case is interesting for comparing the performance of the multiscale approaches. For this case, we consider another heterogeneous permeability coefficient $\kappa(x)$ presented in Fig. \ref{ris:HCC2}, wherein the green region $\kappa(x) = \kappa_1 = 1$ and in the yellow region $\kappa(x) = \kappa_2 = 10^4$.


\begin{figure}[!hbt]
\begin{center}
\includegraphics[width=0.4\linewidth]{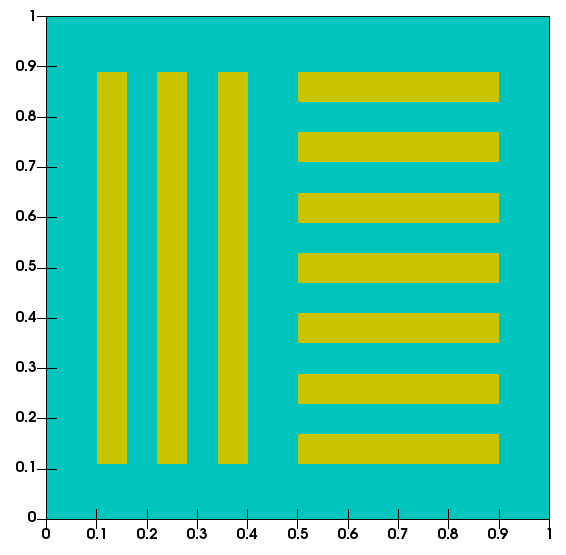}
\end{center}
\caption{High-contrast permeability coefficient $\kappa(x)$ for the semilinear case, where $\kappa_1 = 1$ is green and $\kappa_2 = 10^4$ is yellow.}
\label{ris:HCC2}
\end{figure}


Fig. \ref{ris:sol2} presents the pressure distributions at the final time. From left to right, we depict the reference, MFGMsFEM-FD, and MFGMsFEM-EI solutions, respectively. We use the settings $M = 5$, $N_t = 50$, and $\gamma = 3$ for the multiscale solutions. One can see that the solutions are significantly different from the previous case. The reason for this is the change in the structure of the permeability coefficient and the influence of the nonlinear term. Both multiscale approaches do well in terms of conveying the overall solution dynamics. However, the pressure values of MFGMsFEM-FD are significantly different from the reference ones. In contrast, the values of the MFGMsFEM-EI solution are very close to the reference solution, which indicates that the proposed multiscale approach accurately approximates the reference solution using a few degrees of freedom and time steps.

\begin{figure}[!hbt]
\begin{center}
\begin{minipage}[h]{0.325\linewidth}
\center{\includegraphics[width=\linewidth]{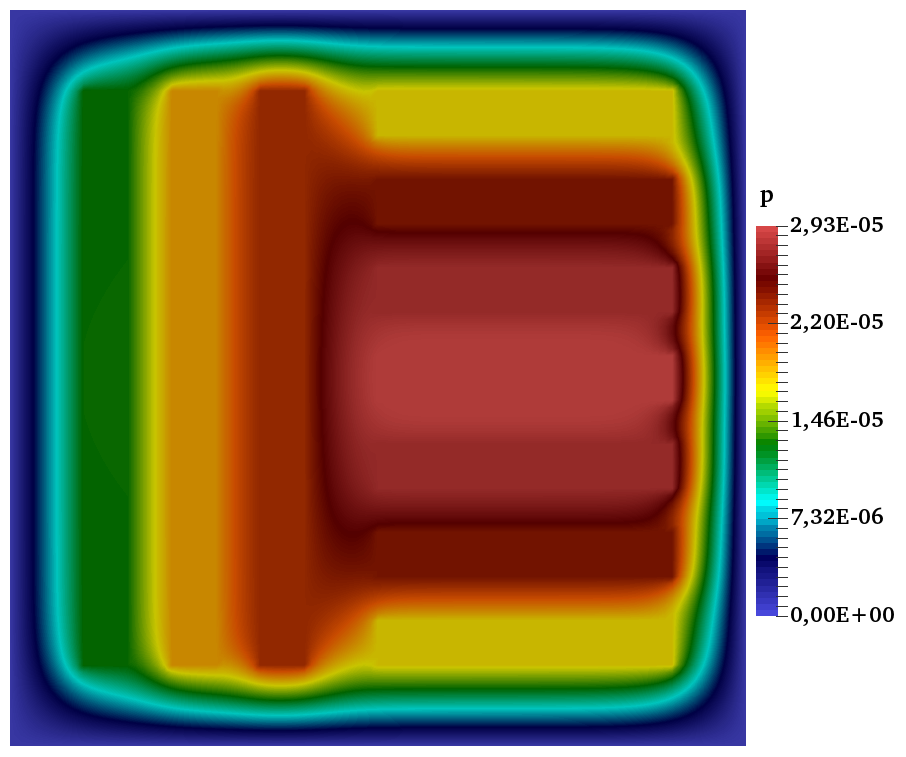}}
\end{minipage}
\begin{minipage}[h]{0.325\linewidth}
\center{\includegraphics[width=\linewidth]{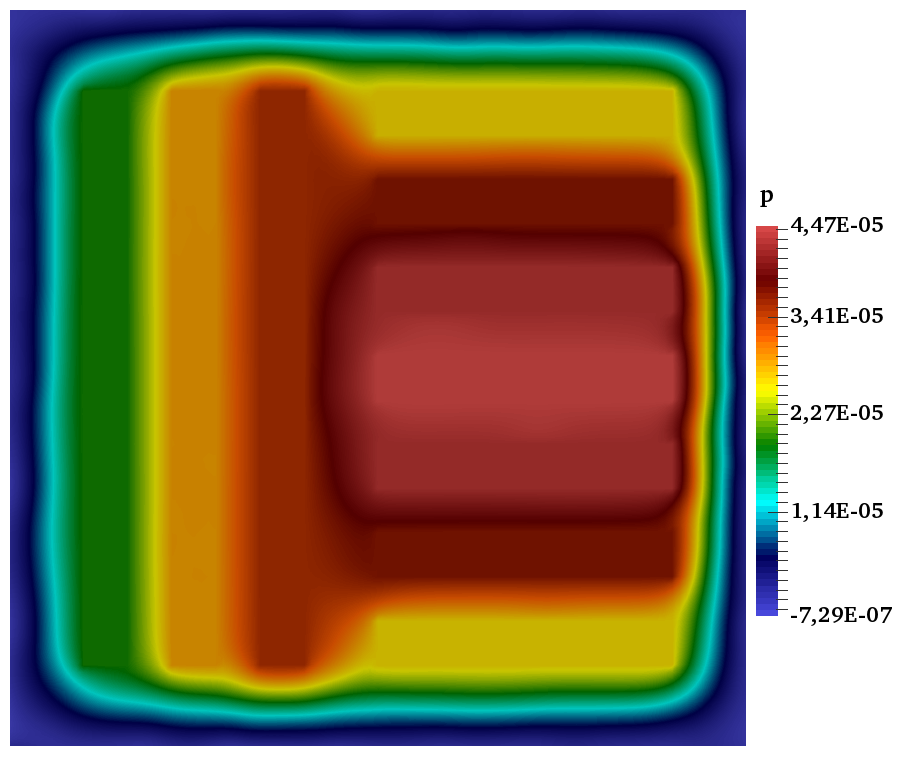}}
\end{minipage}
\begin{minipage}[h]{0.325\linewidth}
\center{\includegraphics[width=\linewidth]{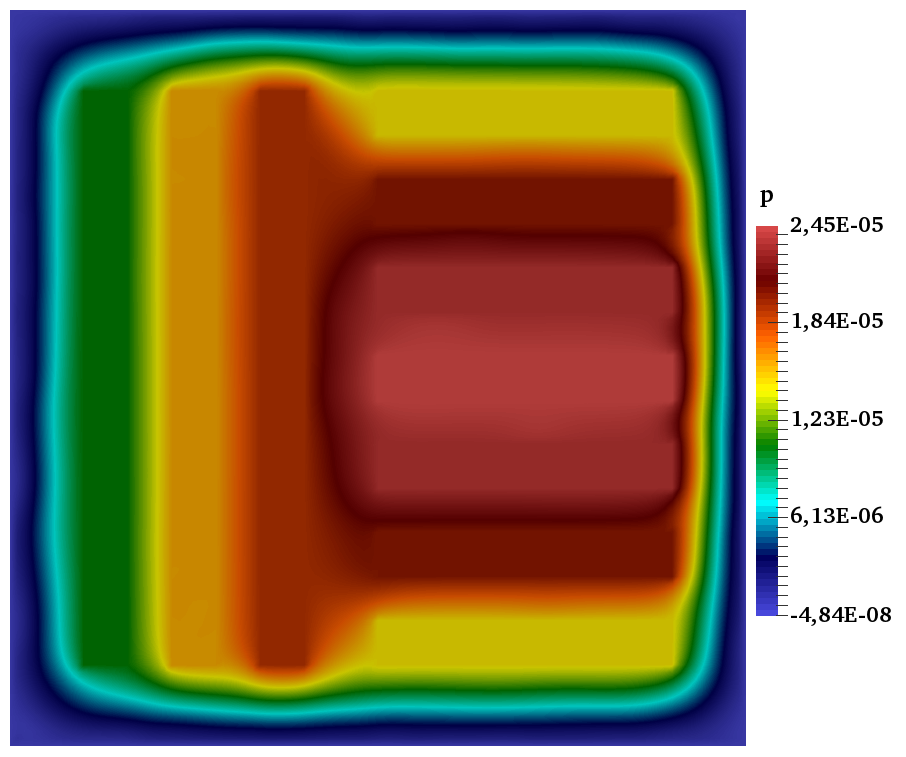}}
\end{minipage}
\end{center}
\caption{Pressure distributions at the final time for the semilinear case. Left: reference solution. Center: MFGMsFEM-FD solution. Right: MFGMsFEM-EI solution. The multiscale approaches' parameters are $M = 5$, $N_t = 50$ and $\gamma = 3$.}
\label{ris:sol2}
\end{figure}

\begin{figure}[!hbt]
\begin{center}
\begin{minipage}[h]{0.44\linewidth}
\center{\includegraphics[width=\linewidth]{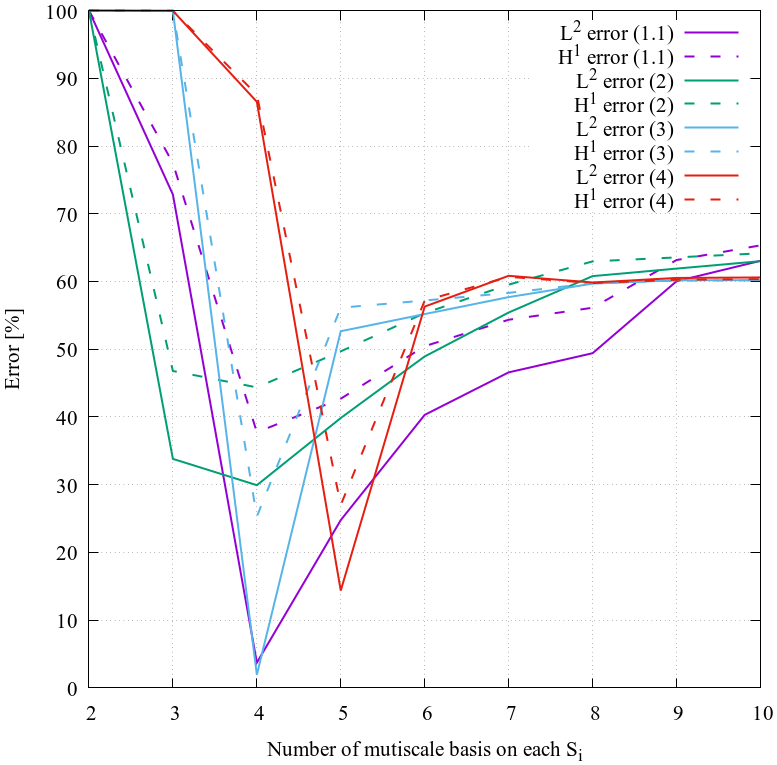}}
\end{minipage}
\begin{minipage}[h]{0.44\linewidth}
\center{\includegraphics[width=\linewidth]{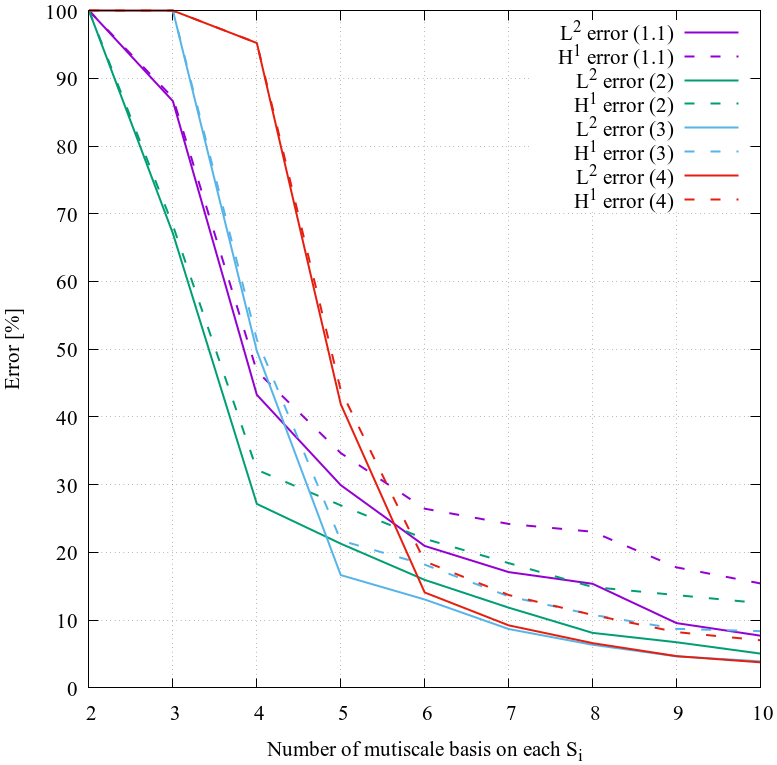}}
\end{minipage}
\end{center}
\caption{The relative $L^2$ and $H^1$ errors between the reference solution and the multiscale solutions using $N_t = 50$ and various $M$ and $\gamma$ at the final time $T = 0.2$ for the semilinear case. Left: MFGMsFEM-FD. Right: MFGMsFEM-EI.}
\label{ris:error21}
\end{figure}

\begin{figure}[!hbt]
\begin{center}
\begin{minipage}[h]{0.44\linewidth}
\center{\includegraphics[width=\linewidth]{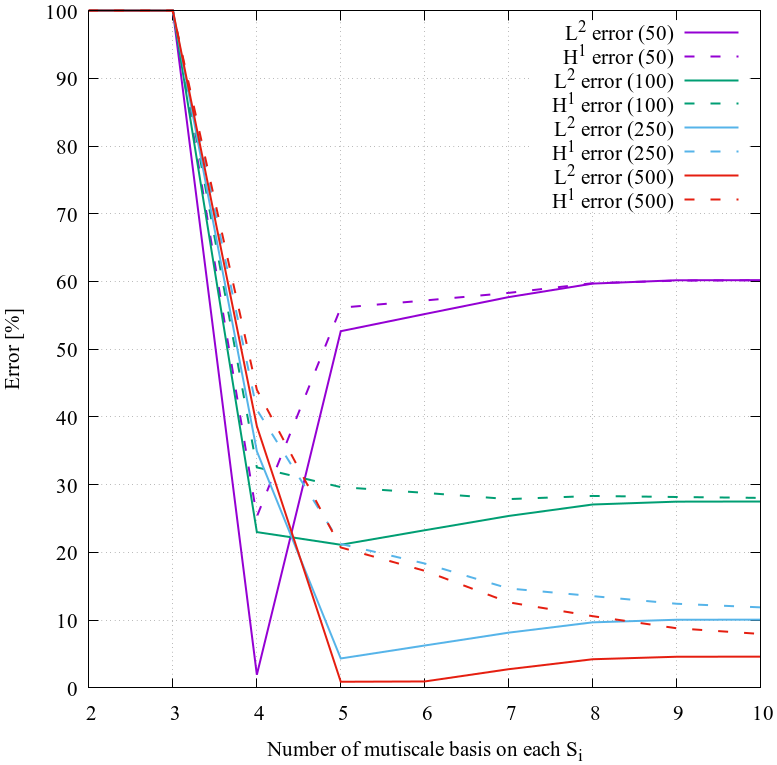}}
\end{minipage}
\begin{minipage}[h]{0.44\linewidth}
\center{\includegraphics[width=\linewidth]{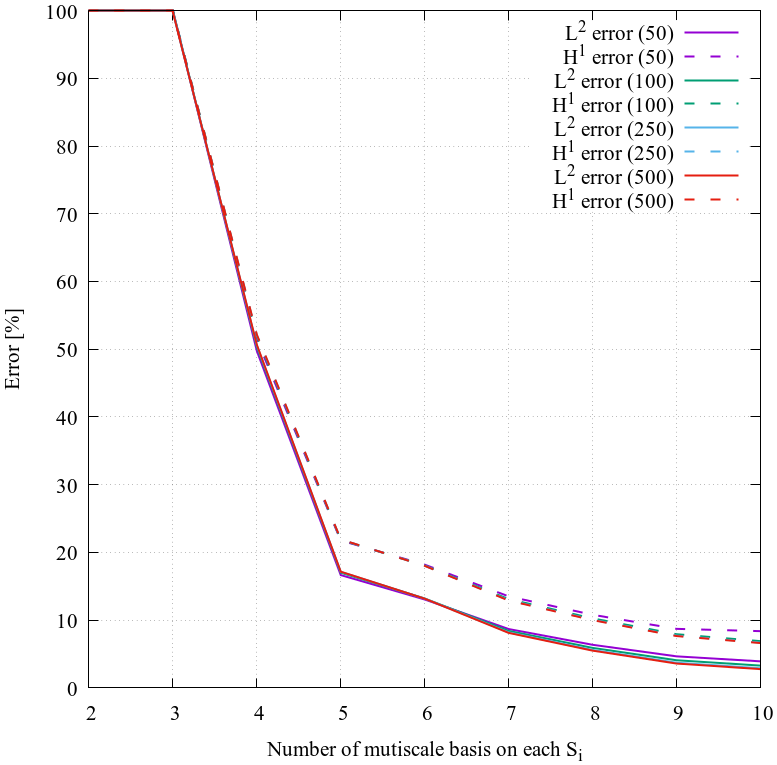}}
\end{minipage}
\end{center}
\caption{The relative $L^2$ and $H^1$ errors between the reference solution and the multiscale solutions using $\gamma = 3$ and various $M$ and $N_t$ at the final time $T = 0.2$ for the semilinear case. Left: MFGMsFEM-FD. Right: MFGMsFEM-EI.}
\label{ris:error22}
\end{figure}

As in the previous case, we analyze the errors of multiscale approaches depending on the coverage parameter $\gamma$. For this purpose, we fix the number of time steps at $N_t = 50$. Fig. \ref{ris:error21} presents plots of the relative $L^2$ and $H^1$ errors for MFGMsFEM-FD (left) and MFGMsFEM-EI (right). Let us consider the errors for MFGMsFEM-FD. We observe similar results as for the linear case. The general dynamics of the $L^2$ and $H^1$ errors are much the same. There is no convergence on the basis functions for all the $\gamma$ cases. Coverage parameter $\gamma$ affects the initial decay of the error, which then grows. Table \ref{tab:tab21} presents the exact values of the relative errors.

Let us consider the relative $L^2$ and $H^1$ errors depending on $\gamma$ for MFGMsFEM-EI. Again, we observe a similar picture as in the linear case. We see an apparent convergence in basis functions for all considered $\gamma$ values. The parameter $\gamma$ affects the rate of error decay. For $\gamma > 1.1$, the larger the value of $\gamma$, the later the decay begins with possible subsequent surpassing of the smaller $\gamma$ errors. Table \ref{tab:tab22} presents the exact values of the relative errors. Therefore, MFGMsFEM-EI provides stable and accurate solutions for various $\gamma$ cases.

Next, we consider the effect of time step size on the errors of multiscale methods. For this purpose, we fix $\gamma = 3$ and vary the number of time steps and basis functions. Figure \ref{ris:error22} presents plots of the relative $L^2$ and $H^1$ errors at the final time for MFGMsFEM-FD and MFGMsFEM-EI (from left to right). The results are similar to those for the linear case. For MFGMsFEM-FD, one can see that a small time step size leads to a decrease in the error. At the same time, we do not observe convergence on the basis functions. The exception is the $H^1$ errors for $N_t \geqslant 250$, which slowly decrease. Table \ref{tab:tab23} shows the exact values of the errors.

Finally, let us analyze the effect of the time step on the MFGMsFEM-EI errors. Compared to the linear case, the convergence rate has slightly decreased. The influence of the nonlinear right-hand side causes this. However, we still clearly see convergence on the basis functions for all time step sizes. One can observe that a smaller time step leads to a slightly more accurate solution, but the difference is insignificant. The overall dynamics of the $L^2$ and $H^1$ errors are largely similar. Table \ref{tab:tab24} shows the exact values of the errors.

\begin{table}[!hbt]
\begin{center}
\begin{tabular}{c| c| c| c| c| c| c| c| c|}
\multirow{2}{*}{$M$} & \multicolumn{2}{c|}{$\gamma = 1.1$} & \multicolumn{2}{c|}{$\gamma = 2$} & \multicolumn{2}{c|}{$\gamma = 3$} & \multicolumn{2}{c|}{$\gamma = 4$} \\
 & ${||e||}_{L^2}$ & ${||e||}_{H^1}$ & ${||e||}_{L^2}$ & ${||e||}_{H^1}$ & ${||e||}_{L^2}$ & ${||e||}_{H^1}$ & ${||e||}_{L^2}$ & ${||e||}_{H^1}$ \\
\hline
2 & 100 & 100 & 100 & 100 & 100 & 100 & 100 & 100 \\
3 & 72.837 & 77.467 & 33.811 & 46.792 & 100 & 100 & 100 & 100 \\
4 & 3.736  & 37.742 & 29.93  & 44.371 & 1.947  & 25.297 & 86.529 & 87.584 \\
5 & 24.778 & 42.703 & 39.849 & 49.69  & 52.663 & 56.136 & 14.382 & 27.005 \\
6 & 40.319 & 50.451 & 48.935 & 55.335 & 55.199 & 57.197 & 56.306 & 57.174 \\
7 & 46.593 & 54.356 & 55.428 & 59.544 & 57.702 & 58.32  & 60.846 & 60.688 \\
8 & 49.417 & 56.126 & 60.797 & 62.967 & 59.692 & 59.75  & 59.854 & 59.769 \\
9 & 60.035 & 63.165 & 61.902 & 63.557 & 60.186 & 60.124 & 60.524 & 60.258 \\
10& 63.055 & 65.337 & 63.007 & 64.167 & 60.204 & 60.188 & 60.602 & 60.356
\end{tabular}
\end{center}
\caption{The relative $L^2$ and $H^1$ errors between the reference solution and the MFGMsFEM-FD solution using $N_t = 50$ and various $M$ and $\gamma$ at the final time $T = 0.2$ for the semilinear case.}
\label{tab:tab21}
\end{table}

\begin{table}[!hbt]
\begin{center}
\begin{tabular}{c| c| c| c| c| c| c| c| c|}
\multirow{2}{*}{$M$} & \multicolumn{2}{c|}{$\gamma = 1.1$} & \multicolumn{2}{c|}{$\gamma = 2$} & \multicolumn{2}{c|}{$\gamma = 3$} & \multicolumn{2}{c|}{$\gamma = 4$} \\
 & ${||e||}_{L^2}$ & ${||e||}_{H^1}$ & ${||e||}_{L^2}$ & ${||e||}_{H^1}$ & ${||e||}_{L^2}$ & ${||e||}_{H^1}$ & ${||e||}_{L^2}$ & ${||e||}_{H^1}$ \\
\hline
2 & 100 & 100 & 100 & 100 & 100 & 100 & 100 & 100 \\
3 & 86.627 & 87.124 & 67.097 & 68.267 & 100 & 100 & 100 & 100 \\
4 & 43.267 & 46.75  & 27.161 & 32.205 & 49.75  & 51.423 & 95.204 & 95.235 \\
5 & 29.938 & 34.678 & 21.295 & 26.959 & 16.657 & 21.77  & 41.892 & 43.958 \\
6 & 20.964 & 26.46  & 15.949 & 22.015 & 13.056 & 18.199 & 14.087 & 18.684 \\
7 & 17.101 & 24.22  & 11.856 & 18.428 & 8.697  & 13.514 & 9.243  & 13.721 \\
8 & 15.378 & 23.06  & 8.124  & 14.901 & 6.366  & 10.797 & 6.611  & 10.774 \\
9 & 9.575 & 17.802  & 6.746  & 13.708 & 4.665  & 8.724  & 4.691  & 8.245  \\
10& 7.714 & 15.428  & 5.071  & 12.522 & 3.928  & 8.387  & 3.781  & 7.043 
\end{tabular}
\end{center}
\caption{The relative $L^2$ and $H^1$ errors between the reference solution and the MFGMsFEM-EI solution using $N_t = 50$ and various $M$ and $\gamma$ at the final time $T = 0.2$ for the semilinear case.}
\label{tab:tab22}
\end{table}

\begin{table}[!hbt]
\begin{center}
\begin{tabular}{c| c| c| c| c| c| c| c| c|}
\multirow{2}{*}{$M$} & \multicolumn{2}{c|}{$N_t = 50$} & \multicolumn{2}{c|}{$N_t = 100$} & \multicolumn{2}{c|}{$N_t = 250$} & \multicolumn{2}{c|}{$N_t = 500$} \\
 & ${||e||}_{L^2}$ & ${||e||}_{H^1}$ & ${||e||}_{L^2}$ & ${||e||}_{H^1}$ & ${||e||}_{L^2}$ & ${||e||}_{H^1}$ & ${||e||}_{L^2}$ & ${||e||}_{H^1}$ \\
\hline
2 & 100 & 100 & 100 & 100 & 100 & 100 & 100 & 100 \\
3 & 100 & 100 & 100 & 100 & 100 & 100 & 100 & 100 \\
4 & 1.947  & 25.297 & 23.006 & 32.555 & 34.946 & 41.075 & 38.642 & 44.021 \\
5 & 52.663 & 56.136 & 21.139 & 29.648 & 4.354  & 21.26  & 0.92   & 20.74  \\
6 & 55.199 & 57.197 & 23.279 & 28.803 & 6.27   & 18.372 & 0.956  & 17.295 \\
7 & 57.702 & 58.32  & 25.39  & 27.88  & 8.16   & 14.688 & 2.768  & 12.638 \\
8 & 59.692 & 59.75  & 27.082 & 28.344 & 9.682  & 13.559 & 4.235  & 10.602 \\
9 & 60.186 & 60.124 & 27.505 & 28.189 & 10.065 & 12.427 & 4.605  & 8.815  \\
10& 60.204 & 60.188 & 27.524 & 28.057 & 10.084 & 11.9   & 4.624  & 7.975  
\end{tabular}
\end{center}
\caption{The relative $L^2$ and $H^1$ errors between the reference solution and the MFGMsFEM-FD solution using $\gamma = 3$ and various $M$ and $N_t$ at the final time $T = 0.2$ for the semilinear case.}
\label{tab:tab23}
\end{table}

\begin{table}[!hbt]
\begin{center}
\begin{tabular}{c| c| c| c| c| c| c| c| c|}
\multirow{2}{*}{$M$} & \multicolumn{2}{c|}{$N_t = 50$} & \multicolumn{2}{c|}{$N_t = 100$} & \multicolumn{2}{c|}{$N_t = 250$} & \multicolumn{2}{c|}{$N_t = 500$} \\
 & ${||e||}_{L^2}$ & ${||e||}_{H^1}$ & ${||e||}_{L^2}$ & ${||e||}_{H^1}$ & ${||e||}_{L^2}$ & ${||e||}_{H^1}$ & ${||e||}_{L^2}$ & ${||e||}_{H^1}$ \\
\hline
2 & 100 & 100 & 100 & 100 & 100 & 100 & 100 & 100 \\
3 & 100 & 100 & 100 & 100 & 100 & 100 & 100 & 100 \\
4 & 49.75  & 51.423 & 50.315 & 51.87  & 50.575 & 52.082 & 50.65  & 52.144 \\
5 & 16.657 & 21.77  & 17.013 & 21.857 & 17.137 & 21.915 & 17.164 & 21.93  \\
6 & 13.056 & 18.2   & 13.194 & 18.077 & 13.193 & 18.038 & 13.179 & 18.025 \\
7 & 8.697  & 13.514 & 8.418  & 13.118 & 8.203  & 12.954 & 8.125  & 12.904 \\
8 & 6.366  & 10.797 & 5.915  & 10.268 & 5.613  & 10.067 & 5.509  & 10.007 \\
9 & 4.665  & 8.724  & 4.089  & 7.935  & 3.725  & 7.723  & 3.602  & 7.663  \\
10& 3.928  & 8.387  & 3.298  & 6.911  & 2.906  & 6.7    & 2.775  & 6.632  
\end{tabular}
\end{center}
\caption{The relative $L^2$ and $H^1$ errors between the reference solution and the MFGMsFEM-EI solution using $\gamma = 3$ and various $M$ and $N_t$ at the final time $T = 0.2$ for the semilinear case.}
\label{tab:tab24}
\end{table}

Thus, MFGMsFEM-EI accurately approximates the reference solution for the semilinear problem with arbitrary $\gamma$ and time step size, ensuring convergence in basis functions. In contrast, MFGMsFEM-FD requires a small time step size to approximate the reference solution accurately, but it does not always ensure convergence on basis functions.

\FloatBarrier
\section{Conclusion}\label{sec:conclusion}
In conclusion, this study presents a novel approach to solving parabolic problems in high-contrast multiscale media by integrating a meshfree generalized multiscale method with exponential integration. This method combines the flexibility and robustness of (multiscale coarse) meshfree techniques with the efficiency of the generalized multiscale finite element methods and the stability advantages of exponential time integration.

We have focused on the coarse meshfree GMsFEM. The coarse space in this method is constructed using a fine-scale computational grid that resolves the heterogeneous parameters of the problem, incorporating multiscale basis functions generated offline using local spectral problems. Specifically, we considered the approximation of solutions to time-dependent semilinear problems in high-contrast multiscale media, addressing stability issues through exponential time integration methods. \corr{Exponential time integration effectively handles stiff problems, allowing for larger time steps and better leveraging GMsFEM's dimension reduction capabilities.}
We presented representative numerical experiments and a detailed analysis of the proposed approach, including a new convergence analysis of the coarse meshfree GMsFEM. 
These advancements contribute to the robustness and efficiency of the coarse mesh-free GMsFEM for various applications.

\section*{Acknowledgments}
The research of Djulustan Nikiforov is supported by the Russian Science Foundation grant No.23-71-10074, \href{https://rscf.ru/en/project/23-71-10074/} {https://rscf.ru/en/project/23-71-10074/}. Dmitry Ammosov would like to acknowledge the support of the Russian government project Science and Universities (project FSRG-2024-0003) aimed at supporting junior laboratories. Yesy Sarmiento gratefully acknowledges DSEP for the International Mobility Program for Postgraduate Students of UNAH, 2024, for supporting her stay at the National University of Colombia to work on GMsFEM methods. Gratitude is also extended to the Department of Mathematics at UNAL for their hospitality.  J. Galvis thanks  MATHDATA - AUIP Network ({Red Iberoamericana de Investigaci\'on en Matem\'aticas Aplicadas a Datos} \href{https://www.mathdata.science/}{https://www.mathdata.science/}), the Center of Excellence in Scientific Computing (Centro de Excelencia en Computaci\'on Cient\'ifica) of the Universidad Nacional de Colombia.

\section{Author Contribution Statement}
All authors contributed equally to the research, development, and writing of this paper. Djulustan Nikiforov, Leonardo A. Poveda, Dmitry Ammosov, Yesy Sarmiento, and Juan Galvis collaboratively developed the conceptual framework, designed the methodology, conducted the analysis, and interpreted the results.  Djulustan Nikiforov and Dmitry Ammosov implemented the numerical test presented in the paper.  Each author participated in drafting and revising the manuscript, ensuring the integrity and accuracy of the work presented. All authors have read and approved the final version of the manuscript.

\appendix
\section{Strengthened Cauchy-Schwarz Inequalities}
Here, we clarify how we obtained \eqref{StrengthenedCS}. We follow \cite{tw}.
The simple abstract inequality is the following.
Assume that 
$$v=\sum_{i=1}^{N_v^H} v_i$$
with $\mbox{supp}(v_i)\subset S_i$. That is, $\{v_i\}_{i=1}^{{N_v^H}}$ is a decomposition of $v$ analogous to the decomposition in  \eqref{eq:decompo}. For the bilinear form $a$, we have 
\[
a(v_i,v_j)\leq \epsilon_{i,j}  a(v_i,v_i)^{1/2} a(v_j,v_j)^{1/2}
\]
for $1\leq i,j,\leq N_v^H$ 
where $\epsilon_{i,j}=1$ if $S_i\cap S_j\not = \emptyset$ and $\epsilon_{i,j}=1$ if $S_i\cap S_j = \emptyset$. 
This set of inequalities is referred to as strengthened Cauchy-Schwarz inequalities.
Define the $N_v^H\times N_v^H$ matrix  $\rho(\epsilon)=[\epsilon_{i,j}]$. Then,

\[
a(v,v)=\sum_{i=1}^{N_v^H}\sum_{i=1}^{N_v^H} a(v_i,v_j) \leq \sum_{i=1}^{N_v^H}\sum_{i=1}^{N_v^H} \epsilon_{i,j}a(v_i,v_i)^{1/2}a(v_j,v_j)^{1/2}
\leq \lambda_{\max}(\rho(\epsilon))
\sum_{i=1}^{N_v^H} a(v_i,v_i).
\]
It is easy to see that $\lambda_{\max}(\rho(\epsilon)) \leq C_{\rm{ov}}$, where $C_{\rm{ov}}$ is defined in \eqref{eq:def:cov}. See Chapters 2 and 3 and Assumptions 2.3, 2.4 and 3.2 in \cite{tw}.
\bibliographystyle{plain}
\bibliography{references}

\end{document}

%% file: point_cloud.tex
\begin{tikzpicture}[scale = 0.7] 
\def\L{10.0}; 
\def\M{10.0}; 
\draw[black] (0,0) rectangle (\L,\M); 
\filldraw[blue] (0*\L,0*\M) circle (2pt); 
\filldraw[blue] (0.18*\L,0*\M) circle (2pt); 
\filldraw[blue] (0.27*\L,0*\M) circle (2pt); 
\filldraw[blue] (0.36*\L,0*\M) circle (2pt); 
\filldraw[blue] (0.44*\L,0*\M) circle (2pt); 
\filldraw[blue] (0.53*\L,0*\M) circle (2pt); 
\filldraw[blue] (0.61*\L,0*\M) circle (2pt); 
\filldraw[blue] (0.71*\L,0*\M) circle (2pt); 
\filldraw[blue] (0.77*\L,0*\M) circle (2pt); 
\filldraw[blue] (0.87*\L,0*\M) circle (2pt); 
\filldraw[blue] (1*\L,0*\M) circle (2pt); 
\filldraw[blue] (0*\L,0.15*\M) circle (2pt); 
\filldraw[blue] (0.13*\L,0.11*\M) circle (2pt); 
\filldraw[blue] (0.23*\L,0.13*\M) circle (2pt); 
\filldraw[blue] (0.34*\L,0.13*\M) circle (2pt); 
\filldraw[blue] (0.42*\L,0.11*\M) circle (2pt); 
\filldraw[blue] (0.51*\L,0.1*\M) circle (2pt); 
\filldraw[blue] (0.6*\L,0.13*\M) circle (2pt); 
\filldraw[blue] (0.68*\L,0.09*\M) circle (2pt); 
\filldraw[blue] (0.8*\L,0.1*\M) circle (2pt); 
\filldraw[blue] (0.9*\L,0.17*\M) circle (2pt); 
\filldraw[blue] (1*\L,0.2*\M) circle (2pt); 
\filldraw[blue] (0*\L,0.21*\M) circle (2pt); 
\filldraw[blue] (0.12*\L,0.23*\M) circle (2pt); 
\filldraw[blue] (0.21*\L,0.24*\M) circle (2pt); 
\filldraw[blue] (0.32*\L,0.22*\M) circle (2pt); 
\filldraw[blue] (0.41*\L,0.2*\M) circle (2pt); 
\filldraw[blue] (0.51*\L,0.18*\M) circle (2pt); 
\filldraw[blue] (0.58*\L,0.24*\M) circle (2pt); 
\filldraw[blue] (0.65*\L,0.23*\M) circle (2pt); 
\filldraw[blue] (0.73*\L,0.18*\M) circle (2pt); 
\filldraw[blue] (0.83*\L,0.23*\M) circle (2pt); 
\filldraw[blue] (1*\L,0.28*\M) circle (2pt); 
\filldraw[blue] (0*\L,0.37*\M) circle (2pt); 
\filldraw[blue] (0.18*\L,0.32*\M) circle (2pt); 
\filldraw[blue] (0.29*\L,0.32*\M) circle (2pt); 
\filldraw[blue] (0.37*\L,0.3*\M) circle (2pt); 
\filldraw[blue] (0.47*\L,0.34*\M) circle (2pt); 
\filldraw[blue] (0.48*\L,0.26*\M) circle (2pt); 
\filldraw[blue] (0.56*\L,0.33*\M) circle (2pt); 
\filldraw[blue] (0.66*\L,0.31*\M) circle (2pt); 
\filldraw[blue] (0.75*\L,0.3*\M) circle (2pt); 
\filldraw[blue] (0.87*\L,0.33*\M) circle (2pt); 
\filldraw[blue] (1*\L,0.36*\M) circle (2pt); 
\filldraw[blue] (0*\L,0.43*\M) circle (2pt); 
\filldraw[blue] (0.13*\L,0.4*\M) circle (2pt); 
\filldraw[blue] (0.25*\L,0.42*\M) circle (2pt); 
\filldraw[blue] (0.35*\L,0.39*\M) circle (2pt); 
\filldraw[blue] (0.43*\L,0.41*\M) circle (2pt); 
\filldraw[blue] (0.51*\L,0.41*\M) circle (2pt); 
\filldraw[blue] (0.59*\L,0.41*\M) circle (2pt); 
\filldraw[blue] (0.68*\L,0.41*\M) circle (2pt); 
\filldraw[blue] (0.77*\L,0.37*\M) circle (2pt); 
\filldraw[blue] (0.85*\L,0.43*\M) circle (2pt); 
\filldraw[blue] (1*\L,0.45*\M) circle (2pt); 
\filldraw[blue] (0*\L,0.53*\M) circle (2pt); 
\filldraw[blue] (0.12*\L,0.52*\M) circle (2pt); 
\filldraw[blue] (0.24*\L,0.54*\M) circle (2pt); 
\filldraw[blue] (0.31*\L,0.49*\M) circle (2pt); 
\filldraw[blue] (0.38*\L,0.49*\M) circle (2pt); 
\filldraw[blue] (0.47*\L,0.5*\M) circle (2pt); 
\filldraw[blue] (0.55*\L,0.5*\M) circle (2pt); 
\filldraw[blue] (0.63*\L,0.5*\M) circle (2pt); 
\filldraw[blue] (0.74*\L,0.48*\M) circle (2pt); 
\filldraw[blue] (0.86*\L,0.52*\M) circle (2pt); 
\filldraw[blue] (1*\L,0.55*\M) circle (2pt); 
\filldraw[blue] (0*\L,0.61*\M) circle (2pt); 
\filldraw[blue] (0.14*\L,0.63*\M) circle (2pt); 
\filldraw[blue] (0.28*\L,0.62*\M) circle (2pt); 
\filldraw[blue] (0.36*\L,0.58*\M) circle (2pt); 
\filldraw[blue] (0.44*\L,0.57*\M) circle (2pt); 
\filldraw[blue] (0.5*\L,0.61*\M) circle (2pt); 
\filldraw[blue] (0.59*\L,0.58*\M) circle (2pt); 
\filldraw[blue] (0.67*\L,0.6*\M) circle (2pt); 
\filldraw[blue] (0.78*\L,0.59*\M) circle (2pt); 
\filldraw[blue] (0.87*\L,0.61*\M) circle (2pt); 
\filldraw[blue] (1*\L,0.65*\M) circle (2pt); 
\filldraw[blue] (0*\L,0.72*\M) circle (2pt); 
\filldraw[blue] (0.12*\L,0.72*\M) circle (2pt); 
\filldraw[blue] (0.25*\L,0.69*\M) circle (2pt); 
\filldraw[blue] (0.36*\L,0.69*\M) circle (2pt); 
\filldraw[blue] (0.43*\L,0.67*\M) circle (2pt); 
\filldraw[blue] (0.49*\L,0.73*\M) circle (2pt); 
\filldraw[blue] (0.57*\L,0.67*\M) circle (2pt); 
\filldraw[blue] (0.66*\L,0.69*\M) circle (2pt); 
\filldraw[blue] (0.76*\L,0.68*\M) circle (2pt); 
\filldraw[blue] (0.87*\L,0.71*\M) circle (2pt); 
\filldraw[blue] (1*\L,0.73*\M) circle (2pt); 
\filldraw[blue] (0*\L,0.83*\M) circle (2pt); 
\filldraw[blue] (0.1*\L,0.8*\M) circle (2pt); 
\filldraw[blue] (0.2*\L,0.77*\M) circle (2pt); 
\filldraw[blue] (0.31*\L,0.79*\M) circle (2pt); 
\filldraw[blue] (0.41*\L,0.79*\M) circle (2pt); 
\filldraw[blue] (0.5*\L,0.82*\M) circle (2pt); 
\filldraw[blue] (0.57*\L,0.76*\M) circle (2pt); 
\filldraw[blue] (0.68*\L,0.79*\M) circle (2pt); 
\filldraw[blue] (0.77*\L,0.77*\M) circle (2pt); 
\filldraw[blue] (0.87*\L,0.8*\M) circle (2pt); 
\filldraw[blue] (1*\L,0.77*\M) circle (2pt); 
\filldraw[blue] (0*\L,0.91*\M) circle (2pt); 
\filldraw[blue] (0.13*\L,0.89*\M) circle (2pt); 
\filldraw[blue] (0.27*\L,0.89*\M) circle (2pt); 
\filldraw[blue] (0.36*\L,0.87*\M) circle (2pt); 
\filldraw[blue] (0.44*\L,0.89*\M) circle (2pt); 
\filldraw[blue] (0.53*\L,0.91*\M) circle (2pt); 
\filldraw[blue] (0.6*\L,0.86*\M) circle (2pt); 
\filldraw[blue] (0.69*\L,0.88*\M) circle (2pt); 
\filldraw[blue] (0.8*\L,0.87*\M) circle (2pt); 
\filldraw[blue] (0.9*\L,0.9*\M) circle (2pt); 
\filldraw[blue] (1*\L,0.84*\M) circle (2pt); 
\filldraw[blue] (0*\L,1*\M) circle (2pt); 
\filldraw[blue] (0.15*\L,1*\M) circle (2pt); 
\filldraw[blue] (0.2*\L,1*\M) circle (2pt); 
\filldraw[blue] (0.33*\L,1*\M) circle (2pt); 
\filldraw[blue] (0.39*\L,1*\M) circle (2pt); 
\filldraw[blue] (0.5*\L,1*\M) circle (2pt); 
\filldraw[blue] (0.57*\L,1*\M) circle (2pt); 
\filldraw[blue] (0.66*\L,1*\M) circle (2pt); 
\filldraw[blue] (0.78*\L,1*\M) circle (2pt); 
\filldraw[blue] (0.82*\L,1*\M) circle (2pt); 
\filldraw[blue] (1*\L,1*\M) circle (2pt); 
\end{tikzpicture}

%% file: main_revised_with_changes_clean.bbl
\begin{thebibliography}{10}

\bibitem{abreu2019convergence}
E.~Abreu, C.~Diaz, and J.~Galvis.
\newblock A convergence analysis of generalized multiscale finite element
  methods.
\newblock {\em Journal of Computational Physics}, 396:303--324, 2019.

\bibitem{con2020}
Eduardo Abreu, Ciro Diaz, Juan Galvis, and John P{\'e}rez.
\newblock On the conservation properties in multiple scale coupling and
  simulation for darcy flow with hyperbolic-transport in complex flows.
\newblock {\em Multiscale Modeling \& Simulation}, 18(4):1375--1408, 2020.

\bibitem{ahrens2005paraview}
James Ahrens, Berk Geveci, and Charles Law.
\newblock Paraview: An end-user tool for large data visualization.
\newblock {\em The visualization handbook}, 717(8), 2005.

\bibitem{bastian2013simulation}
Peter Bastian, Johannes Kraus, Robert Scheichl, and Mary Wheeler.
\newblock {\em Simulation of flow in porous media: applications in energy and
  environment}, volume~12.
\newblock Walter de Gruyter, 2013.

\bibitem{battiato2019theory}
Ilenia Battiato, Peter~T Ferrero~V, Daniel O’Malley, Cass~T Miller, Pawan~S
  Takhar, Francisco~J Vald{\'e}s-Parada, and Brian~D Wood.
\newblock Theory and applications of macroscale models in porous media.
\newblock {\em Transport in Porous Media}, 130(1):5--76, 2019.

\bibitem{Expint2007}
H\r{a}vard Berland, B\r{a}rd Skaflestad, and Will~M. Wright.
\newblock Expint---a matlab package for exponential integrators.
\newblock {\em ACM Trans. Math. Softw.}, 33(1):4–es, mar 2007.

\bibitem{calo2016randomized}
Victor~M Calo, Yalchin Efendiev, Juan Galvis, and Guanglian Li.
\newblock Randomized oversampling for generalized multiscale finite element
  methods.
\newblock {\em Multiscale Modeling \& Simulation}, 14(1):482--501, 2016.

\bibitem{calvo2024robust}
Juan~G Calvo and Juan Galvis.
\newblock Robust domain decomposition methods for high-contrast multiscale
  problems on irregular domains with virtual element discretizations.
\newblock {\em Journal of Computational Physics}, 505:112909, 2024.

\bibitem{chung2014generalized}
Eric~T Chung, Yalchin Efendiev, and Shubin Fu.
\newblock Generalized multiscale finite element method for elasticity
  equations.
\newblock {\em GEM-International Journal on Geomathematics}, 5:225--254, 2014.

\bibitem{chung2018constraint}
Eric~T. Chung, Yalchin Efendiev, and Wing~Tat Leung.
\newblock Constraint energy minimizing generalized multiscale finite element
  method.
\newblock {\em Comput. Methods Appl. Mech. Engrg.}, 339:298--319, 2018.

\bibitem{chung2014adaptive}
Eric~T Chung, Yalchin Efendiev, and Guanglian Li.
\newblock An adaptive gmsfem for high-contrast flow problems.
\newblock {\em Journal of Computational Physics}, 273:54--76, 2014.

\bibitem{contreras2023exponential}
LF~Contreras, D~Pardo, E~Abreu, J~Mu{\~n}oz-Matute, C~Diaz, and J~Galvis.
\newblock An exponential integration generalized multiscale finite element
  method for parabolic problems.
\newblock {\em Journal of Computational Physics}, 479:112014, 2023.

\bibitem{dixon1986weak}
J.~Dixon and S.~McKee.
\newblock Weakly singular discrete {G}ronwall inequalities.
\newblock {\em Z. Angew. Math. Mech.}, 66(11):535--544, 1986.

\bibitem{djulustan2024meshfree}
Nikiforov Djulustan and Stepanov Sergei.
\newblock Meshfree multiscale method with partially explicit time
  discretization for nonlinear stefan problem.
\newblock {\em Journal of Computational and Applied Mathematics}, page 116020,
  2024.

\bibitem{du2002meshfree}
Qiang Du, Max Gunzburger, and Lili Ju.
\newblock Meshfree, probabilistic determination of point sets and support
  regions for meshless computing.
\newblock {\em Computer methods in applied mechanics and engineering},
  191(13-14):1349--1366, 2002.

\bibitem{abreu_recursiveparabolic_2021}
L.~Macul E.~Abreu1, P.~Ferraz.
\newblock A multiscale recursive numerical method for semilinear parabolic
  problems.
\newblock {\em CILAMCE, PANACM}, 2021.

\bibitem{egh12}
Y.~Efendiev, J.~Galvis, and T.~Hou.
\newblock Generalized multiscale finite element methods.
\newblock {\em Journal of Computational Physics}, 251:116--135, 2013.

\bibitem{egw10}
Y.~Efendiev, J.~Galvis, and X.H. Wu.
\newblock Multiscale finite element methods for high-contrast problems using
  local spectral basis functions.
\newblock {\em Journal of Computational Physics}, 230:937--955, 2011.

\bibitem{efendiev2011multiscale}
Yalchin Efendiev, Juan Galvis, and Xiao-Hui Wu.
\newblock Multiscale finite element methods for high-contrast problems using
  local spectral basis functions.
\newblock {\em Journal of Computational Physics}, 230(4):937--955, 2011.

\bibitem{fu2019generalized}
Shubin Fu, Eric Chung, and Tina Mai.
\newblock Generalized multiscale finite element method for a strain-limiting
  nonlinear elasticity model.
\newblock {\em Journal of Computational and Applied Mathematics}, 359:153--165,
  2019.

\bibitem{ge09_1}
J.~Galvis and Y.~Efendiev.
\newblock Domain decomposition preconditioners for multiscale flows in high
  contrast media.
\newblock {\em SIAM J. Multiscale Modeling and Simulation}, 8:1461--1483, 2010.

\bibitem{ge09_1reduceddim}
J.~Galvis and Y.~Efendiev.
\newblock Domain decomposition preconditioners for multiscale flows in high
  contrast media. {R}educed dimensional coarse spaces.
\newblock {\em SIAM J. Multiscale Modeling and Simulation}, 8:1621--1644, 2010.

\bibitem{gao2015generalized}
Kai Gao, Shubin Fu, Richard~L Gibson~Jr, Eric~T Chung, and Yalchin Efendiev.
\newblock Generalized multiscale finite-element method (gmsfem) for elastic
  wave propagation in heterogeneous, anisotropic media.
\newblock {\em Journal of Computational Physics}, 295:161--188, 2015.

\bibitem{henry1981geometric}
Daniel Henry.
\newblock {\em Geometric theory of semilinear parabolic equations}, volume 840
  of {\em Lecture Notes in Mathematics}.
\newblock Springer-Verlag, Berlin-New York, 1981.

\bibitem{hochbruck2005explicit}
Marlis Hochbruck and Alexander Ostermann.
\newblock Explicit exponential {R}unge-{K}utta methods for semilinear parabolic
  problems.
\newblock {\em SIAM J. Numer. Anal.}, 43(3):1069--1090, 2005.

\bibitem{hochbruck2010exponential}
Marlis Hochbruck and Alexander Ostermann.
\newblock Exponential integrators.
\newblock {\em Acta Numerica}, 19:209--286, 2010.

\bibitem{Huang_exponetial_parabolic_2019}
Jianguo Huang, Lili Ju, and Bo~Wu.
\newblock A fast compact exponential time differencing method for semilinear
  parabolic equations with neumann boundary conditions.
\newblock {\em Applied Mathematics Letters}, 94:257--265, 2019.

\bibitem{li2022partially}
Wenyuan Li, Anatoly Alikhanov, Yalchin Efendiev, and Wing~Tat Leung.
\newblock Partially explicit time discretization for nonlinear time fractional
  diffusion equations.
\newblock {\em Communications in Nonlinear Science and Numerical Simulation},
  113:106440, 2022.

\bibitem{logg2012automated}
Anders Logg, Kent-Andre Mardal, and Garth Wells.
\newblock {\em Automated solution of differential equations by the finite
  element method: The FEniCS book}, volume~84.
\newblock Springer Science \& Business Media, 2012.

\bibitem{narayanan2018flow}
Natarajan Narayanan, Berlin Mohanadhas, and Vasudevan Mangottiri.
\newblock {\em Flow and Transport in Subsurface Environment}.
\newblock Springer, 2018.

\bibitem{nikiforov2023meshfree}
Djulustan Nikiforov.
\newblock Meshfree generalized multiscale finite element method.
\newblock {\em Journal of Computational Physics}, 474:111798, 2023.

\bibitem{nikiforov2024meshfree}
Djulustan Nikiforov.
\newblock Meshfree multiscale method with partially explicit time
  discretization.
\newblock {\em arXiv preprint arXiv:2402.02763}, 2024.

\bibitem{nikiforov2024meshfree1}
Djulustan Nikiforov, Sergei Stepanov, and Nyurgun Lazarev.
\newblock Meshfree multiscale method for the infiltration problem in
  permafrost.
\newblock {\em Journal of Computational and Applied Mathematics}, 449:115988,
  2024.

\bibitem{nikiforov2023modeling}
DY~Nikiforov and SP~Stepanov.
\newblock Modeling of artificial ground freezing using a meshfree multiscale
  method.
\newblock {\em Lobachevskii Journal of Mathematics}, 44(3):1206--1214, 2023.

\bibitem{poveda2024second}
Leonardo~A. Poveda, Juan Galvis, and Eric Chung.
\newblock A second-order exponential integration constraint energy minimizing
  generalized multiscale method for parabolic problems.
\newblock {\em J. Comput. Phys.}, 502:Paper No. 112796, 2024.

\bibitem{SUN201876}
Zheng Sun, José~A. Carrillo, and Chi-Wang Shu.
\newblock A discontinuous galerkin method for nonlinear parabolic equations and
  gradient flow problems with interaction potentials.
\newblock {\em Journal of Computational Physics}, 352:76--104, 2018.

\bibitem{thomee2006galerkin}
Vidar Thom\'{e}e.
\newblock {\em Galerkin finite element methods for parabolic problems},
  volume~25 of {\em Springer Series in Computational Mathematics}.
\newblock Springer-Verlag, Berlin, second edition, 2006.

\bibitem{tw}
A.~Toselli and O.~Widlund.
\newblock {\em Domain decomposition methods -- {A}lgorithms and {T}heory},
  volume~34 of {\em Computational Mathematics}.
\newblock Springer-Verlag, 2005.

\bibitem{zambrano2021fast}
Miguel Zambrano, Sintya Serrano, Boyan~S Lazarov, and Juan Galvis.
\newblock Fast multiscale contrast independent preconditioners for linear
  elastic topology optimization problems.
\newblock {\em Journal of Computational and Applied Mathematics}, 389:113366,
  2021.

\end{thebibliography}
